\documentclass[twoside,12pt]{amsart}
\usepackage{amsmath,latexsym,amssymb,mathptm, times,verbatim, enumerate,mathcomp,lscape}
    \usepackage{longtable}\usepackage{tikz,stmaryrd}
\usepackage{pdfpages}
\usepackage{calligra}
\usepackage{calrsfs}
\usepackage[mathscr]{euscript}
\input amssym.def
\input amssym 
\input xypic
\input xy
\xyoption{all}
\usepackage[right=1.5in,left=1.5in,top=1in,bottom=1in]{geometry}

\def\mfK{\mathfrak K}
\def\Pbar{\overline{P}}
\def\XXX{\rho}
\def\Lbar{\ov{L}}
\def\id{\operatorname{id}}

\def\e{\epsilon}

\def\DGGamma{\operatorname{DG}\Gamma}
\def\DG{\operatorname{DG}}
\def\comult{\operatorname{comult}}
\def\proj{\operatorname{proj}}

\def \ms{\medskip}

\def\bw{\bigwedge}

\def\ov{\overline}

\def\mfp{\mathfrak p}
\def\M{\mathfrak M}

\def\rank{\operatorname{rank}}

\def\p{\oplus}
\def\im{\operatorname{im}}

\def\m{\mathfrak m}

\def\a{\alpha}

\def\Hom{\operatorname{Hom}}
\def\kk {\pmb k}

\def\Tor{\operatorname{Tor}}
\def\t{\otimes}
\def\w{\wedge}
\def\HH{\operatorname{H}}

\def\Ext{\operatorname{Ext}}
\def\ts{\textstyle}
\def\Ass{\operatorname{Ass}}
\def\mfq{\mathfrak q}

\newtheorem{theorem}{Theorem}[section]
\newtheorem{lemma}[theorem]{Lemma}

\newtheorem{proposition-no-advance}[equation]{Proposition}

\newtheorem{claim-no-advance}[equation]{Claim}
\newtheorem{observation}[theorem]{Observation}

\newtheorem{quick consequences}[theorem]{Quick Consequences}

\theoremstyle{definition}
\newtheorem{definitions and conventions}[theorem]{Definition and Conventions}

\newtheorem{definition-no-advance}[equation]{Definition}

\newtheorem{facts and definitions}[theorem]{Facts and Definitions}
\newtheorem{definition}[theorem]{Definition}

\newtheorem{remark-no-advance}[equation]{Remark}
\newtheorem{remarks-no-advance}[equation]{Remarks}
\newtheorem{remarks}[theorem]{Remarks}
\newtheorem{data}[theorem]{Data}

\newtheorem{careful calculation}[theorem]{Careful Calculation}

\newtheorem{present summary}[theorem]{Present Summary}

\newtheorem{example}[theorem]{Example}

\newtheorem{further reductions}[theorem]{Further Reductions}

\newtheorem{chunk}[theorem]{}
\newtheorem{chunk-no-advance}[equation]{}
\newtheorem{subchunk}[equation]{}

\newtheorem{marching orders}[theorem]{Marching Orders}

\newtheorem{circle the wagons}[theorem]{Circle the wagons}

\numberwithin{equation}{theorem}

\begin{document}

\baselineskip=16pt

\title[Use $\DG$-methods  to build a  matrix factorization
]
{Use $\DG$-methods  to build a  matrix factorization
}

\date{\today}

\author[A.~R.~Kustin]{Andrew R.~Kustin}
\address{Andrew R.~Kustin\\ Department of Mathematics\\ University of South Carolina\\\newline 
Columbia\\ SC 29208\\ U.S.A.} \email{kustin@math.sc.edu}



\thanks{}

\begin{abstract}  Let $P$ be a commutative Noetherian ring, $\mfK$ be an ideal of $P$ which is   generated by a regular sequence of length four, $f$ be a regular element of $P$, and $\Pbar$ be the hypersurface ring $P/(f)$. Assume that $\mfK:f$ is a grade four Gorenstein ideal of $P$. We give a resolution $N$ of $\Pbar/\mfK\Pbar$ 
by free $\Pbar$-modules. 
 The resolution $N$ is built from a Differential Graded Algebra resolution of $P/(\mfK:f)$ by free $P$-modules, together with one homotopy map. 
In particular, we give an explicit form for the matrix factorization which is the infinite tail of the resolution $N$.
\end{abstract}

\maketitle

\section{Introduction.}\label{19}
Let $P$ be a commutative Noetherian ring, $\mfK$ be an ideal of $P$, $f$ be a regular element of $P$, and $\Pbar$ be the hypersurface ring $P/(f)$. This paper grew out of a desire to find an efficient method for resolving $\Pbar/\mfK\Pbar$ by free $\Pbar$-modules. We are particularly interested in this problem when $\mfK$ is generated by a regular sequence. 

The ultimate goal is to compare the resolution of the 
Frobenius
 powers $\Pbar/\mfK^{[q]}\Pbar$ to the resolution of $\Pbar/\mfK\Pbar$, for $q=p^e$, where $P$ is a ring of  prime characteristic $p$. The most interesting feature of the $\Pbar$-resolution of $\Pbar/\mfK^{[q]}\Pbar$ is the infinite tail of the resolution, which is a matrix factorization of $f$. One goal is to determine the number of infinite tails that appear as $q=p^e$ varies and the least positive value of $e'$  for which the infinite tail of the resolution of $\Pbar/\mfK^{[qq']}\Pbar$ is isomorphic to the infinite tail of the resolution of $\Pbar/\mfK^{[q]}\Pbar$,  with  $q'=p^{e'}$.

This ultimate goal has been 
 accomplished 
 when $P=\kk[x,y,z]$, $\mfK$ is the maximal ideal $(x,y,z)$, and $\kk$ is a field of characteristic $p$.  If $f=x^n+y^n+z^n$, 
then the Betti numbers of $\Pbar/\mfK^{[q]}\Pbar$ are calculated in \cite{KuU} and the resolution of  $\Pbar/\mfK^{[q]}\Pbar$ is given in \cite{KRV}. If 
$f$ is a generic homogeneous form of $P$, 
 then the graded Betti numbers of $\Pbar/\mfK^{[q]}\Pbar$ are calculated in \cite{MRR}.

The present paper gives a resolution $N$ of $\Pbar/\mfK\Pbar$ 
by free $\Pbar$-modules when $\mfK$ is generated by a regular sequence of length four, $(\mfK:f)$ is a Gorenstein ideal of grade four in $P$,  and $P$ is an arbitrary commutative Noetherian ring. The resolution $N$ is built from a Differential Graded Algebra resolution $$M:\quad 0\to M_4\xrightarrow{m_4} M_3\xrightarrow{m_3} M_2\xrightarrow{m_2} M_1\xrightarrow{m_1} M_0=P$$ of $P/(\mfK:f)$ together with a homotopy map $X:M_1\to M_2$. The 
 resolution $N$ is given in Theorem~\ref{thm17}. 
The matrix factorization which comprises the infinite tail of $N$ is given in Theorem~\ref{17.9}. The precise properties of the homotopy map $X:M_1\to M_2$ are also given in Theorem~\ref{17.9}. The most important part of the paper is the proof that $X$ exists. This proof is given in Sections~\ref{5}, \ref{6}, and \ref{7}.  

The cleanest version of the matrix factorization of Theorem~\ref{17.9} occurs when $f$ is the element $\beta_0(1)$ of $P$ which corresponds to the product 
$$\a_1(\e_1)\cdot \a_1(\e_2)\cdot \a_1(\e_3)\cdot \a_1(\e_4)$$ in $M_4$, where
$K=\bw^\bullet(\bigoplus P\e_i)$ is a Koszul complex which resolves $P/\mfK$, and $$\a:K\to M$$ is a map of $\DGGamma$-algebras. In this case, the matrix factorization of $f$ is given by
$$\bmatrix X|_{M_{1,2}}&\a_2& m_3|_{M_{3,2}}\endbmatrix \quad \text{and}\quad
\bmatrix \proj_{M_{1,2}}\circ m_2\\\beta_2\\\proj_{M_{3,2}}\circ X^\dagger \endbmatrix.$$
The decompositions $M_1=M_{1,1}\p M_{1,2}$ and $M_3=M_{3,1}\p M_{3,2}$ are explained in the text, $\a_2$ is the degree two component of $\a$, and $\beta_2$ and $X^\dagger$ essentially are maps adjoint to $\a_2$ and $X$, respectively. An arbitrary $f$ has the form $r\beta_0(1)+\kappa$, where $r\in P$ and $\kappa\in \mfK$. Once one has a matrix factorization for $\beta_0(1)$, then there is no added difficulty in finding the matrix factorization for an  arbitrary $f$; but the formulas become more complicated. In particular, a streamlined version of the paper can be read if 
one takes $r=1$ and $\sigma$, $z_i$, $w_i$, $Y$, and $W$ all to be zero.

In Section~\ref{10} we describe two other interpretations of the map $X:M_1\to M_2$. On the one hand, $X$ is a higher order multiplication in the sense of \cite{Sl93,Ku-aci-dg}. On the other hand, $X$ and its adjoint give a homotopy from the complex $M$ to itself.

\tableofcontents

\section{Notation, conventions, and elementary results.}\label{Prelims}

\begin{chunk}The {\it grade} of a proper ideal $I$ in a commutative Noetherian ring $P$ is the length of the
longest regular sequence on $P$ in $I$. The ideal $I$ of $P$ is called {\it
perfect} if the grade of $I$ is equal to the projective dimension
 of the $P$-module $P/I$. The grade $g$ ideal $I$ is called {\it Gorenstein} if it is
perfect and $\Ext_{P}^{g}(P/I,P)\cong P/I.$ It follows from Bass \cite[Prop.~5.1]{Bss} that if $I$ is a Gorenstein ideal in a Gorenstein ring $P$, then $P/I$ is also a Gorenstein ring. 
\end{chunk}

\begin{chunk} 
 A complex $\cdots \to F_2\to F_1\to F_0\to 0$ is called {\it acyclic} if the only non-zero homology occurs in position zero. \end{chunk}

\begin{chunk}\label{2.1} Let $P$ be a commutative Noetherian ring, 
  $X$ be a free $P$-module, and $Y$ be a $P$-module. 
The
rules for a divided power algebra $D_\bullet X$ are recorded in \cite[Def.~1.7.1]{GL}  or  \cite[Appendix 2]{Ei95}.
(In practice these rules say that $x^{
(a)}$ behaves like $x^
a/(a!)$ would behave if $a!$ were
a unit in $P$.) Two rules that we use often are
\begin{align}
\label{2.2.a} (px)^{(n)}&=p^nx^{(n)},&&\text{for $p\in P$ and $x\in X$, and}\\ 
(x+y)^{(n)}&=\sum_{i=0}^n x^{(i)}y^{(n-i)},&&\text{for  $x,y\in X$.}\notag\end{align}
If $x$ and $x'$ are elements of $X$, then $x\cdot x'=x'\cdot x$ in $D_2(X)$. The co-multiplication homomorphism $$\comult:D_2X\to X\t_PX$$ sends $x^{(2)}$ to $x\t x$ and sends $x\cdot x'$ to $x\t x'+x'\t x$, for $x,x'$ in $X$. Often we will describe a homomorphism $\phi: D_2X\to Y$ by giving the value of $\phi(x^{(2)})$ for each $x\in X$. One then automatically knows the value of $\phi(x\cdot x')$, for $x,x'\in X$ because $$(x+x')^{(2)}=x^{(2)}+x\cdot x'+{x'}^{(2)}.$$ \end{chunk}

\begin{chunk}
If $P$ is a ring and $A$, $B$, and $C$ are $P$-modules, then the $P$-module homomorphism $\phi: A\otimes_P B\to C$ is a {\it perfect pairing} if the induced $P$-module homomorphisms $A\to \mathrm{Hom}_P(B,C)$ and $B\to \mathrm{Hom}_P(A,C)$,  given by $a\mapsto \phi(a\otimes \underline{\phantom{X}})$ and $b\mapsto \phi(\underline{\phantom{X}}\otimes b)$, respectively, are isomorphisms.
\end{chunk}

\begin{chunk}
\label{2.4}
A {\it Differential Graded  
algebra} 
$F$ (written \underline{$\DG$-algebra}) over the commutative Noetherian ring $P$ is a complex of finitely generated free $P$-modules $(F, d)$:
$$\cdots 
\xrightarrow{d_2}F_1\xrightarrow{d_1}F_0=P,$$
together with a unitary, associative  multiplication $F\t_PF\to F$, which satisfies
\begin{enumerate}[\rm(a)]
\item $F_iF_j\subseteq F_{i+j}$,
\item\label{2.4.b} $d_{i+j}(x_ix_j)=d_i(x_i)x_j+(-1)^i x_id_j(x_j)$,
\item $x_ix_j=(-1)^{ij}x_jx_i$, and
\item $x_i^2=0$, when  $i$ is odd,
\end{enumerate}
for  $x_{\ell}\in F_{\ell}$. 
The $\DG$-algebra $F$ is called a \underline{$\DGGamma$-algebra} {\rm(}or a $\DG$-algebra with divided powers{\rm)} if, for each positive even index $i$ and each element $x_i$ of $F_i$, there is a family of elements $\{x_i^{(k)}\}$ which satisfy the divided power axioms of \ref{2.1}, 
 and which also satisfy
\begin{equation}\label{2.3.1}d_{ik}(x_i^{(k)})=d_i(x_i) x_i^{(k-1)}.\end{equation}
The $\DG$-algebra $F$ 
 exhibits {\it Poincar\'e duality} if there there is an integer $m$ such that $F_i=0$ for $m< i$, $F_m$ is isomorphic to $P$, and for each integer $i$, the multiplication map $$F_i\t_{P} F_{m-i}\to F_m$$ is a perfect pairing of $P$-modules.
\end{chunk}

\begin{example}
The Koszul complex is the prototype of a $\DGGamma$-algebra which exhibits Poincar\'e duality. \end{example}

Lemma~\ref{apr29-19} is used at a critical spot in 
 the proof of Lemma~\ref{17.34}. The assertion is obvious if $P$ is a local ring or if $P$ is  a domain; however the assertion holds without any hypothesis imposed on $P$. 

\begin{lemma}
\label{apr29-19} Let $P$ be a commutative Noetherian ring and $\mfK$ be an ideal in $P$ which is generated by a regular sequence, then there exists a regular sequence  $a_1,\dots,a_n$ in $\mfK$ which generates $\mfK$ with the property that each $a_i$ is a regular element of $P$.\end{lemma}

\begin{proof} Let $a_1,\dots,a_n$ be a regular sequence which generates $\mfK$. Observe that for any choice of $p_2,\dots,p_n$ in $P$, the elements
$a_1,a_2+p_2a_1,\dots,a_n+p_na_1$  also form a regular sequence 
which generates $\mfK$. 
Fix an integer  $i$ with $2\le i\le n$. We prove there exists an element $p_i\in P$ with $a_i+p_ia_1$  a regular element of $P$.
 Let
 $$S=\{\mfp\in \Ass (P)\mid \text{ $\mfp$ is not properly contained in $\mfq$ for any $\mfq\in \Ass(P)$}\}.$$
The point is that the set of zero divisors of $P$ is $\cup_{\mfp\in S}\, \mfp$ and no prime of $S$ contains another prime of $S$. 
Decompose $S$ into two subsets:
$$S_1=\{\mfp\in S\mid a_i\in \mfp\}\quad\text{and}\quad  S_2=\{\mfp\in S\mid a_i\notin \mfp\}.$$ If $\mfp\in S_2$, then $\mfp\not\subseteq \mfq$ for any $\mfq$ of $S_1$. Thus, the prime avoidance lemma ensures that  $\mfp\not\subseteq \cup_{\mfq\in S_1}\mfq$ and there exists an element $p_\mfp\in \mfp\setminus \cup_{\mfq\in S_1}\mfq$. Observe that
$$\ts a_i+(\prod_{\mfp\in S_2}p_{\mfp}) a_1$$ is a regular element on $P$.
\end{proof}

\section{Matrix factorization.}

\begin{data}\label{data17} Let $P$ be a commutative Noetherian ring,   $f$ be a regular element in $P$, 
 $\mfK$ be an ideal of $P$ which is generated by a regular sequence of length  
 four, 
and \begin{equation}\label{17.45}M:\quad 0\to M_4\xrightarrow{m_4} M_3 \xrightarrow {m_3} M_2 \xrightarrow
{m_2} M_1 \xrightarrow{m_1} M_0=P\end{equation}
 be a complex of length four which is a resolution of $P/(\mfK:f)$ by free $P$-modules. Assume that
\begin{enumerate}[\rm(a)]
\item\label{data17.a} $M$ is a 
 $\DGGamma$-algebra which exhibits Poincar\'e duality, and 
\item \label{data17.b}
the module $M_1$ is the direct sum of two free submodules $$M_1=M_{1,1}\p M_{1,2},$$ with $\rank M_{1,1}=4$ and $m_1(M_{1,1})=\mfK$. 
\end{enumerate}\end{data}

\begin{remarks}\label{Nov-22}

$ $

\begin{enumerate}[\rm(a)]

\item\label{nov-22} According to \cite{Ku-DG}, every self-dual resolution $$M:\quad 0\to M_4\to M_3\to M_2\to M_1\to M_0=P$$ 
 is a $\DGGamma$-algebra which exhibits Poincar\'e duality. Earlier versions of this theorem \cite{KM80,Ku87} proved the result when $P$ is Gorenstein and local and $M$ is a minimal resolution. It is shown in \cite{Ku-DG} that these three hypotheses are unnecessary.

\item It is not important for our purposes that the resolution $M$ of $P/(\mfK:f)$ be a minimal resolution (when this notion is defined). Indeed,  hypothesis \ref{data17}.(\ref{data17.b}) might rule out the possibility of $M$ being a minimal resolution. Nonetheless, the result of \cite{Ku-DG} may be applied in order to obtain a resolution $M$ which satisfies both hypotheses \ref{data17}.(\ref{data17.a}) and \ref{data17}.(\ref{data17.b}).
\end{enumerate}
\end{remarks}

There are  three results in this paper. 
 Theorem \ref{17.9} gives an explicit matrix factorization of $f$ provided there 
exists a map $X:M_1\to M_2$ which satisfies five properties;
Theorem~\ref{17.11} states   that the map $X$ exists; and Theorem~\ref{thm17} 
states  
that the matrix factorization of Theorem~\ref{17.9} induces   the infinite  tail in the resolution of $P/(f,\mfK)$ by free $P/(f)$-modules.

\begin{definitions and conventions}\label{17.42} Adopt Data~{\rm\ref{data17}}.
Let \begin{equation}K:\quad 0\to K_4\xrightarrow{k_4} K_3 \xrightarrow {k_3} K_2 \xrightarrow
{k_2} K_1 \xrightarrow{k_1} K_0=P\label{17.48}\end{equation} be the Koszul complex which is a resolution of $P/\mfK$. Notice that $K$ is automatically a $\DGGamma$-algebra which exhibits Poincar\'e duality. The elements of $K_i$ are denoted by $\phi_i$, and $$[-]_K:K_4\to P$$ is a fixed orientation isomorphism. The elements of $M_i$  are denoted by $\theta_i$ and $$[-]_M:M_4\to P$$ is a fixed orientation isomorphism.
Define$$\a_0:K_0=P\to M_0=P$$ to be the identity map and define 
$$\a_1:K_1\to M_1=M_{1,1}\p M_{1,2}$$ so that
$\proj_{M_{1,1}}\circ \a_1: K_1\to M_{1,1}$ is the  isomorphism
for which the diagram
\begin{equation}\xymatrix{K_1\ar[rr]^{k_1}\ar[d]_{\proj_{M_{1,1}}\circ \a_1}^{\cong}&&\mfK\\M_{1,1}\ar[urr]_{m_1|_{M_{1,1}}}}\label{5-23.1}\end{equation} commutes; and 
\begin{equation}\proj_{M_{1,2}}\circ \a_1: K_1\to M_{1,2}\quad \text{is the zero map.} \label{17.46}\end{equation} 
(Recall the decomposition of $M_1$ which is described in \ref{data17}.(\ref{data17.b}).)
Define $\a:K\to M$ to be the map of $\DGGamma$-algebras which extends $\a_0$ and $\a_1$.
 
  Define 
$$\beta_i:M_i\to K_i$$ by
\begin{equation}\label{17.2}[\beta_i(\theta_i)\wedge  \phi_{4-i}]_K=[\theta_i\cdot \a_{4-i}(\phi_{4-i})]_M,\end{equation}
for all $\theta_i\in M_i$ and $\phi_{4-i}\in K_{4-i}$. (The fact that $M$ and $K$ are Poincar\'e duality algebras ensures that this definition is meaningful.) \end{definitions and conventions}
 
\begin{subchunk}
\label{17.3.2}Notice that linkage theory 
guarantees that 
$$\mfK:\im m_1=(\mfK, \beta_0(1))\quad\text{and}\quad 
\mfK:\beta_0(1) = \im m_1.$$ On the other hand, linkage theory also guarantees that
$$\mfK:\im m_1=(\mfK, f)\quad\text{and}\quad 
\mfK:f = \im m_1.$$So \begin{equation}\label{oct28} f=r\beta_0(1)+k_1(\sigma)\end{equation} for some $r\in P$ and $\sigma\in K_1$. Usually, $r$ will be a unit in $P$; indeed, for example, if $P$ is a local ring, then $r$ is a unit.
\end{subchunk}

\begin{subchunk}\label{17.7}  Define submodules 
$$M_{3,1}=\{\theta_3\in M_3\mid \theta_3 M_{1,2}=0\}\quad \text{and} \quad M_{3,2}=\{\theta_3\in M_3\mid \theta_3 M_{1,1}=0\}$$ of $M_3$.
The assumption that $M$ is a Poincar\'e duality algebra  
 ensures that 
$$M_3= M_{3,1}\p M_{3,2}$$ and that the multiplication maps
$$M_{1,1}\t_PM_{3,1}\to M_4\quad\text{and}\quad M_{1,2}\t_PM_{3,2}\to M_4$$ are both perfect pairings.
\end{subchunk}

\begin{subchunk}\label{19.3.7}The maps $\beta_4:M_4\to K_4$ and 
$\beta_3|_{M_{3,1}}:M_{3,1}\to K_3$ are isomorphisms. Indeed, the definitions yield that
$$[\beta_4(\theta_4)]_K=[\beta_4(\theta_4)\w 1]_K=[\theta_4\cdot\a_0(1)]_M=[\theta_4]_M,$$ and that  $[-]_K$ and $[-]_M$
are both isomorphisms. Similarly,  the definitions yield that the map
$$M_{3,1}\to \Hom_P(K_1,P),\quad\text{given by}\quad \theta_{3,1}\mapsto [\beta_3(\theta_{3,1})\w -]_K,$$ is an isomorphism. 
 On the other hand, $K$ is a Poincar\'e Duality algebra; so,
$$K_3\to\Hom_P(K_1,P),\quad\text{given by}\quad \phi_{3}\mapsto [\phi_3\w -]_K,$$ is also an isomorphism.
It follows that $\beta_3|_{M_{3,1}}:M_{3,1}\to K_3$ is an isomorphism.\end{subchunk}

\begin{subchunk}
\label{17.8}For each homomorphism $h:M_1\to M_2$, let $$h^\dagger:M_2\to M_3$$ be the homomorphism  
defined by 
 \begin{equation}\notag h^\dagger(\theta_2)\cdot \theta_1=\theta_2\cdot h(\theta_1),\end{equation}for
$\theta_i\in M_i$.
(The existence of $h^\dagger$ is also guaranteed by the assumption that $M$ is a Poincar\'e duality algebra.)\end{subchunk} 

\begin{subchunk}\label{17.8'}The homomorphisms
$$z_i:K_i\to K_{i+1}\quad\text{and}\quad w_i:M_i\to M_{i+1}$$ are defined by
$$z_i(\phi_i)=\phi_i\w \sigma\quad \text{and}\quad w_i(\theta_i)=\theta_i\cdot \a_1(\sigma),$$
for $\phi_i\in K_i$ and $\theta_i\in M_i$.
The homomorphisms 
$$Y:M_2\to K_2\quad\text{and}\quad 
W: K_2\to M_2$$
are defined by
\begin{align*}Y&{}=z_1\circ (\proj_{M_{1,1}}\circ \a_1)^{-1}\circ 
\proj_{M_{1,1}}\circ m_2\quad\text{and}\\
W&{}=m_3\circ(\beta_3|_{M_{3,1}})^{-1}\circ z_2.\end{align*}
Recall from (\ref{5-23.1}) and \ref{19.3.7} that the indicated inverse maps exist.\end{subchunk}

\bigskip
We are now able to state the  result about matrix factorization. This result 
gives an explicit matrix factorization of the $f$ of Data~\ref{data17} in terms of the maps defined in Definition~\ref{17.42} and one other map $X:M_1\to M_2$ provided 
the map $X$  exists and satisfies   five properties. Theorem~\ref{17.11} states   that the map $X$ exists; and Theorem~\ref{thm17} 
states  
that the matrix factorization of Theorem~\ref{17.9} induces   the infinite  tail in the resolution of $P/(f,\mfK)$ by free $P/(f)$-modules.   
Recall from \ref{17.3.2} that the parameter $r$ is usually a unit. In this case, there is no reason to consider the matrix factorization M.F.\ref{19.9.1}. Indeed, in this case, M.F.\ref{19.9.2} is obtained from M.F.\ref{19.9.1} by splitting off a trivial factorization. Furthermore, as was observed in the Introduction, a streamlined, but still meaningful, version of the paper can be read if 
one takes $r=1$ and $\sigma$, $z_i$, $w_i$, $Y$, and $W$ all to be zero.
\begin{theorem}\label{17.9}Adopt the language of Definition~{\rm\ref{17.42}}.  Suppose that $$X:M_1\to M_2$$ is an $R$-module homomorphism which satisfies
\begin{enumerate}[\rm(a)]
\item\label{17.9.a} $X\circ \a_1=0$,
\item\label{17.9.b} $m_2\circ X=\beta_0(1)\cdot \id_{M_1}-\alpha_1\circ\beta_1$,
\item\label{17.9.c} $X\circ m_2+ m_3\circ X^\dagger=\beta_0(1)\cdot \id_{M_2}-\a_2\circ \beta_2$,
\item\label{17.9.d} $X^\dagger\circ X=0$, and
\item\label{17.9.e} $X^\dagger\circ \a_2=0$.
\end{enumerate}
Then the following statements hold.
\begin{enumerate}[\rm M.F. 1.]

\item\label{19.9.1} Let $G_{\text{\rm even}}$ and $G_{\text{\rm odd}}$ be the free $P$-modules
$$G_{\text{\rm even}}=M_{1,2}\p K_2\p M_{3}
\p  K_4\quad
\text{and}\quad 
G_{\text{\rm odd}}=
 M_{2}\p K_3
\p M_4$$
and $g_{\text{\rm even}}: G_{\text{\rm even}}\to  G_{\text{\rm odd}}$ and
$g_{\text{\rm odd}}: G_{\text{\rm odd}}\to  G_{\text{\rm even}}$ be the $P$-module homomorphisms
$$g_{\text{\rm even}}=
 \bmatrix 
(rX-w_1)|_{M_{1,2}}&\a_2&m_3&0\\
0&-z_2&r\beta_3&-k_4\\
0&0&-w_3&\a_4
\endbmatrix\quad
\text{and}$$
$$g_{\text{\rm odd}}=
\bmatrix 
\proj_{M_{1,2}}\circ m_2&0&0 \\
r\beta_2-Y&-k_3&0\\
rX^\dagger+w_2
&\a_3
&m_4\\
0&z_3&r\beta_4
\endbmatrix.$$
Then  the equalities $$g_{\text{\rm odd}}\circ g_{\text{\rm even}}=f \cdot \id_{G_{\text{\rm even}}}\quad \text{and}\quad g_{\text{\rm even}}\circ g_{\text{\rm odd}}=f \cdot \id_{G_{\text{\rm odd}}}$$ hold. 

\item\label{19.9.2} Assume that $r$ is a unit. Let $\acute{G}_{\text{\rm even}}$ and $\acute{G}_{\text{\rm odd}}$ be the free $P$-modules  $$\acute{G}_{\text{\rm even}}=M_{1,2}\p K_2\p  M_{3,2}\quad
\text{and}\quad 
\acute{G}_{\text{\rm odd}}=
 M_{2}$$
$\acute{g}_{\text{\rm even}}: \acute{G}_{\text{\rm even}}\to  \acute{G}_{\text{\rm odd}}$ and
$\acute{g}_{\text{\rm odd}}: \acute{G}_{\text{\rm odd}}\to  \acute{G}_{\text{\rm even}}$ be the $P$-module homomorphisms
$$\acute{g}_{\text{\rm even}}=\bmatrix 
(rX-w_1)|_{M_{1,2}}
&\a_2+r^{-1}W
&m_3|_{M_{3,2}}
 \endbmatrix$$
$$ \acute{g}_{\text{\rm odd}}=\bmatrix 
\proj_{M_{1,2}}\circ m_2\\
r\beta_2-Y\\
\proj_{M_{3,2}}\circ(rX^\dagger+w_2)
\endbmatrix.$$
Then  the equalities $$\acute{g}_{\text{\rm odd}}\circ \acute{g}_{\text{\rm even}}=f \cdot \id_{\acute{G}_{\text{\rm even}}}\quad \text{and}\quad \acute{g}_{\text{\rm even}}\circ \acute{g}_{\text{\rm odd}}=f \cdot \id_{\acute{G}_{\text{\rm odd}}}$$ hold. 

\end{enumerate}
\end{theorem}

\begin{remarks}

$ $

\begin{enumerate}[\rm(a)]

\item The proof of Theorem \ref{17.9} is given in \ref{proof17.9}. First we make numerous preliminary calculations involving the maps of Data~\ref{data17}, Definition~\ref{17.42}, and $X$ itself. 

\item It should be noted that $G_{\text{\rm even}}$ and $G_{\text{\rm odd}}$ have the same rank. Indeed, $K_2$, $K_3$, $K_4$, and $M_4$ have rank $6$, $4$, $1$, and $1$, respectively; $M_1$ and $M_3$ have the same rank; $\rank M_2=2\rank M_1-2$; and $\rank M_{1,2}=\rank M_1-4$. Similarly, 
\begin{align*}\rank \acute{G}_{\text{\rm even}}={}&6+2(\rank M_1-4)=2\rank M_1-2=\rank M_2\\{}={}&\rank\acute{G}_{\text{\rm odd}}.\end{align*}
\end{enumerate}
\end{remarks}
\section{Preliminary calculations.}\label{PC}

In this section we prove many formulas involving the data of \ref{data17} and  \ref{17.42}. 
These formulas 
are used in the  the proof of the existence of $X$ and in the proof of Theorem~\ref{17.9}. There are many of these formulas; but each proof is straightforward. The hard work is involved in the proof of Theorem~\ref{17.11}, where we establish the existence of $X$.

\begin{chunk}We often use the graded product rule on $M$ and $K$ in the following form. If $\theta_j\in M_j$ and $\phi_j\in K_j$
then $0=m_5(\theta_i\cdot \theta_{5-i})$, $0=k_5(\phi_i\w \phi_{5-i})$; and therefore,
\begin{equation}\label{17.5}
\begin{array}{rcl}
0&{}={}&m_i(\theta_i)\cdot \theta_{5-i}+(-1)^i\theta_i\cdot m_{5-i}(\theta_{5-i})\text{ and }\\[5pt] 
0&{}={}&k_i(\phi_i)\w \phi_{5-i}+(-1)^i\phi_i\w k_{5-i}(\phi_{5-i}).\end{array}\end{equation}
\end{chunk}
 \begin{observation}\label{17.41} In the language of {\rm\ref{17.42}}, the maps  $\beta_i$ form a map of complexes.
\end{observation}
\begin{proof} It suffices to show that
$$\beta_i\circ m_{i+1}=k_{i+1}\circ \beta_{i+1}.$$
It suffices to show that
$$[(\beta_i\circ m_{i+1})(\theta_{i+1})\w \phi_{4-i}]_K=[(k_{i+1}\circ \beta_{i+1})(\theta_{i+1})\w \phi_{4-i}]_K.$$
We compute \begingroup\allowdisplaybreaks
\begin{align*}
&[(\beta_i\circ m_{i+1})(\theta_{i+1})\w \phi_{4-i}]_K\\
{}={}&[m_{i+1}(\theta_{i+1})\cdot \a_{4-i}(\phi_{4-i})]_M,&&\text{by (\ref{17.2}),}\\
{}={}&(-1)^i[\theta_{i+1}\cdot (m_{4-i}\circ\a_{4-i})(\phi_{4-i})]_M,&&
\text{by (\ref{17.5}),}\\
{}={}&(-1)^i[\theta_{i+1}\cdot (\a_{3-i}\circ k_{4-i})(\phi_{4-i})]_M,&&\text{since $\a$ is a map of complexes,}\\
{}={}&(-1)^i[(\beta_{i+1}\theta_{i+1})\wedge  k_{4-i}(\phi_{4-i})]_K,&&\text{by (\ref{17.2}),}\\
{}={}&[(k_{i+1}\circ\beta_{i+1})(\theta_{i+1})\wedge  \phi_{4-i}]_K,&&\text{by (\ref{17.5}).}
\end{align*}\endgroup
\vskip-24pt\end{proof}

It is convenient to combine the maps of complexes $\a$ and $\beta$ into the following commutative diagram:
\begin{equation}\label{17.6}
\xymatrix{
0\ar[r]& 
K_4\ar[r]^{k_4}\ar[d]^{\a_4}&
 K_3\ar[r]^{k_3}\ar[d]^{\a_3}& 
K_2\ar[r]^{k_2}\ar
[d]^{\a_2}& K_1   \ar[r]^{k_1}\ar[d]^{\a_1}& K_0\ar[d]^{\a_0}_{=}
\\
0\ar[r]& M_4\ar[r]^{m_4}\ar[d]^{\beta_4}& M_3\ar[r]^{m_3}\ar[d]^{\beta_3}& M_2\ar[r]
^{m_2}\ar[d]^{\beta_2}& M_1   \ar[r]^{m_1}\ar[d]^{\beta_1}& M_0\ar[d]^{\beta_0}
\\
0\ar[r]& K_4\ar[r]^{k_4}\ar[d]^{\a_4}& K_3\ar[r]^{k_3}\ar[d]^{\a_3}& K_2\ar[r]^{k_2}\ar[d]^{\a_2}& K_1   \ar[r]^{k_1}\ar[d]^{\a_1}& K_0\ar[d]^{\a_0}_{=}\\
0\ar[r]& M_4\ar[r]^{m_4}& M_3\ar[r]^{m_3}& M_2\ar[r]
^{m_2}& M_1   \ar[r]^{m_1}& M_0.
}\end{equation}

\begin{observation}
\label{17.3}Adopt the language of {\rm\ref{17.42}}. The following formulas hold for $\theta_\ell$ in $M_\ell$ and $\phi_\ell$ in $K_\ell${\rm:}
\begin{enumerate}[\rm(a)]
\item\label{17.3.a}$\beta_i\circ \a_i=\beta_0(1)\cdot \id_{K_i}$\,, for $0\le i\le 4$,
\item\label{17.3.b} $\theta_i\cdot (\a_{4-i}\circ \beta_{4-i})(\theta_{4-i})=(\a_i\circ \beta_i)(\theta_i)\cdot \theta_{4-i}$\,, for $0\le i\le 4$, 
\item\label{17.3.c} $\beta_i(\theta_j\cdot \a_{i-j}(\phi_{i-j}))=\beta_j(\theta_j)\w \phi_{i-j}$\,, for $0\le j\le i\le 4$,

\item\label{18.D} $\beta_3|_{M_{3,2}}=0$, 

\item\label{18.I}$(\beta_3|_{M_{3,1}})^{-1}\circ \beta_3=\proj_{M_{3,1}}$, and

\item \label{18.N}$w_3\circ (\beta_3|_{M_{3,1}})^{-1}\circ z_2=0$.
\end{enumerate}
\end{observation}
\begin{proof} $ $

\noindent(\ref{17.3.a}) It suffices to show that $$[(\beta_i\circ \a_i)(\phi_i)\w \phi_{4-i}]_K=
[\beta_0(1)\cdot\phi_i\w \phi_{4-i}]_K,$$for all $\phi_i\in  K_i$ and $\phi_{4-i}\in  K_{4-i}$. Observe that 

\begingroup\allowdisplaybreaks
\begin{align*} &[(\beta_i\circ \a_i)(\phi_i)\w \phi_{4-i}]_K\\{}={}&[\a_i(\phi_i)\cdot \a_{4-i}(\phi_{4-i})]_M,&&\text{by \rm(\ref{17.2}),}\\
{}={}&[\a_{4}(\phi_i\w \phi_{4-i})]_M,&&\text{because $\a$ is an algebra map,}\\
{}={}&[1\cdot \a_{4}(\phi_i\w \phi_{4-i})]_M\\
{}={}&[\beta_0(1)\cdot\phi_i\w \phi_{4-i}]_K,&&\text{by \rm(\ref{17.2}).}\end{align*}
\endgroup

\ms \noindent (\ref{17.3.b}) Apply (\ref{17.2}) and graded commutativity multiple times:
\begin{align*}[\theta_i\cdot (\a_{4-i}\circ \beta_{4-i})(\theta_{4-i})]_M
{}={}&[\beta_i(\theta_i)\w \beta_{4-i}(\theta_{4-i})]_K\\
{}={}&(-1)^i[\beta_{4-i}(\theta_{4-i})\w \beta_i(\theta_i)]_K\\
{}={}&(-1)^i[\theta_{4-i}\cdot (\a_{i}\circ\beta_{i})(\theta_i)]_M\\
{}={}&[(\a_{i}\circ\beta_{i})(\theta_i)\cdot  \theta_{4-i}]_M.\end{align*}

\ms \noindent (\ref{17.3.c}) Observe that
\begin{align*} 
&[\beta_i(\theta_j\cdot \a_{i-j}(\phi_{i-j}))\w \phi_{4-i}]_K\\
{}={}&[\theta_j\cdot \a_{i-j}(\phi_{i-j})\cdot \a_{4-i}(\phi_{4-i})]_M,&&\text{by (\ref{17.2})},\\
{}={}&[\theta_j\cdot \a_{4-j}(\phi_{i-j}\w\phi_{4-i})]_M,&&\text{because $\a$ is an algebra map,}\\
{}={}&[\beta_j(\theta_j)\w \phi_{i-j}\w\phi_{4-i}]_K,&&\text{by (\ref{17.2})}.
\end{align*}
Multiplication is associative  in both $K$ and $M$.

\ms\noindent (\ref{18.D}) If $\theta_{3,2}$ is an element of $M_{3,2}$, then 
$$[\beta_3(\theta_{3,2})\w \phi_1]_K=[\theta_{3,2}\cdot \a_1(\phi_1)]_M=0,$$ for all $\phi_1\in K_1$ by the definition of $M_{3,2}$ (see \ref{17.7})); hence $\beta_3(\theta_{3,2})=0$.

\ms \noindent (\ref{18.I})
Observe that
\begin{align*}(\beta_3|_{M_{3,1}})^{-1}\circ \beta_3={}&
(\beta_3|_{M_{3,1}})^{-1}\circ \beta_3\circ(\proj_{M_{3,1}}+\proj_{M_{3,2}})\\
={}&(\beta_3|_{M_{3,1}})^{-1}\circ (\beta_3|_{M_{3,1}}\circ\proj_{M_{3,1}}+\beta_3|_{M_{3,2}}
\circ\proj_{M_{3,2}}).\intertext{Recall from (\ref{18.D}) that $\beta_3|_{M_{3,2}}=0$. Conclude that}(\beta_3|_{M_{3,1}})^{-1}\circ \beta_3={}&(\beta_3|_{M_{3,1}})^{-1}\circ (\beta_3|_{M_{3,1}}\circ\proj_{M_{3,1}})=\proj_{M_{3,1}}.\end{align*}

\ms \noindent (\ref{18.N})
If $\phi_2\in K_2$, then
\begin{align*}
[\big(w_3\circ (\beta_3|_{M_{3,1}})^{-1}\circ z_2\big)(\phi_2)]_M={}&
[(\beta_3|_{M_{3,1}})^{-1}(\phi_2\w \sigma) \cdot \a_1(\sigma)]_M&&\text{by \ref{17.8'}}\\
{}={}&
[\beta_3\big((\beta_3|_{M_{3,1}})^{-1}(\phi_2\w \sigma)\big) \w \sigma]_K&&\text{by \ref{17.2}}\\
{}={}&
[\phi_2\w \sigma\w \sigma]_K=0.
\end{align*}

\end{proof}

\begin{lemma}\label{17*.5}In the language of {\rm\ref{17.42}}, 
 $\ker m_3\cap \ker \beta_3=0$. \end{lemma}

\begin{proof}  Let $\theta_3$ be an element of $\ker m_3\cap \ker \beta_3$.
The complex $M$ is a resolution; so $\theta_3=m_4(\theta_4)$ for some $\theta_4\in M_4$. Apply (\ref{17.6}) to see that $$0=\beta_3(\theta_3)=(\beta_3\circ m_4)(\theta_4)=k_4\circ \beta_4(\theta_4).$$ The map $k_4$ is an injection; consequently, $\beta_4(\theta_4)=0$. On the other hand,
$$0=[\beta_4(\theta_4)]_K=[\beta_4(\theta_4)\w 1]_K=[\theta_4\cdot \a_0(1)]_M=[\theta_4]_M.$$ Thus, $\theta_4=0$ and $\theta_3=m_4(\theta_4)$ is also zero. \end{proof}

\begin{observation}\label{17*.4}The maps and modules 
$$\xymatrix{
0\ar[r]&K_0\ar[r]^{z_0}\ar[d]^{\a_0}&K_1\ar[r]^{z_1}\ar[d]^{\a_1}&K_2\ar[r]^{z_2}\ar[d]^{\a_2}&K_3\ar[r]^{z_3}\ar[d]^{\a_3}&K_4\ar[d]^{\a_4}\\
0\ar[r]&M_0\ar[r]^{w_0}\ar[d]^{\beta_0}&M_1\ar[r]^{w_1}\ar[d]^{\beta_1}&M_2\ar[r]^{w_2}\ar[d]^{\beta_2}&M_3\ar[r]^{w_3}\ar[d]^{\beta_3}&M_4\ar[d]^{\beta_4}\\
0\ar[r]&K_0\ar[r]^{z_0}&K_1\ar[r]^{z_1}&K_2\ar[r]^{z_2}&K_3\ar[r]^{z_3}&K_4
}$$ form  maps of complexes.
\end{observation}
\begin{proof} The maps $z_i$ and $w_i$ are defined in (\ref{17.8'}). Elements of degree
 $1$ in a DG-algebra square to zero. It follows that $z_i\circ z_{i-1}=0$ and $w_i\circ w_{i-1}=0$.
To see that $\a$ is a map of complexes, we observe that 
\begin{align*}(\a_{i+1}\circ z_i)(\phi_i){}={}&\a_{i+1}(\phi_i\w \sigma),&&\text{by the definition of $z$,}\\{}={}&\a_{i}(\phi_i)\cdot \a_1(\sigma),&&\text{because $\a$ is an algebra map,}\\{}={}&(w_i\circ \a_i)(\phi_i),&&\text{by the definition of $w$.}\end{align*}
To see that $\beta$ is a map of complexes, we observe that 
\begin{align*}&[(z_i\circ \beta_i)(\theta_i)\w \phi_{3-i}]_K\\{}={}&
[\beta_i(\theta_i)\w \sigma \w \phi_{3-i}]_K,&&\text{by the definition of $z$,}
\\{}={}&[\theta_i\cdot \a_{4-i}(\sigma \w \phi_{3-i})]_M,&&\text{by (\ref{17.2}),}
\\{}={}&[(\theta_i\cdot \a_1(\sigma)) \cdot\a_{3-i}(\phi_{3-i})]_M,&&\text{because $\a$ is an algebra map,}
\\{}={}&[w_i(\theta_i) \cdot\a_{3-i}(\phi_{3-i})]_M,&&\text{by the definition of $w$,}
\\{}={}&[(\beta_{i+1}\circ w_i)(\theta_i) \w \phi_{3-i}]_K,&&\text{by (\ref{17.2}).}\end{align*}Of course, we used that the multiplication on $M$ is associative.
\end{proof}

\begin{observation}\label{17*.3} The maps $z_i$ and $w_i$ of Definition~{\rm\ref{17.8'}} satisfy the following formulas{\rm:}
\begin{enumerate}[\rm(a)]\item\label{17*.3.a} 
$z_{i-1}\circ k_i-k_{i+1}\circ 
z_i=(-1)^{i+1}k_1(\sigma)\cdot \id_{K_i}$\,, and

\item\label{17*.3.b} $w_{i-1}\circ m_i-m_{i+1}\circ w_i=(-1)^{i+1}k_1(\sigma)\cdot \id_{M_i}$.
\end{enumerate}
\end{observation}

\begin{proof} One uses the definition of $z$ and $w$, the graded product rule, and the commutativity of (\ref{17.6}). (\ref{17*.3.a}) If $\phi_i\in K_i$, then 
\begin{align*}(z_{i-1}\circ k_i-k_{i+1}\circ z_i)(\phi_i)
= {}&k_i(\phi_i)\w \sigma -k_{i+1}(\phi_i\w \sigma )=(-1)^{i+1}k_1(\sigma)\cdot \phi_i
.\end{align*}

\ms \noindent (\ref{17*.3.b}) If $\theta_i\in M_i$, then  \begin{align*}(w_{i-1}\circ m_i-m_{i+1}\circ w_i)(\theta_i)
&{}=m_i(\theta_i)\cdot \a_1(\sigma) -m_{i+1}(\theta_i\cdot \a_1(\sigma) )\\&{}=(-1)^{i+1}k_1(\sigma)\cdot \theta_i.\end{align*}
\end{proof}

\section{There exists a homomorphism $X$ which satisfies {\rm\ref{17.9}.(\ref{17.9.b})} and {\rm\ref{17.9}.(\ref{17.9.c})}.}\label{5}

Retain the data of \ref{data17} and  \ref{17.42}. In this section we produce a formal complex $B$ which automatically has a partial multiplicative structure.  We also produce  a null homotopic map of complexes $c:B\to K$. 
 Our first approximation of the map $X:M_1\to M_2$ is manufactured from this homotopy. This version of $X$ satisfies {\rm\ref{17.9}.(\ref{17.9.b})} and {\rm\ref{17.9}.(\ref{17.9.c})}.

Our inspiration for using $B$ comes from the proof of \cite[Prop.~1.1]{BE77} and from \cite[Sect.~2]{T57}. A complex similar to $B$ plays a crucial role in 
 \cite[Lem.~3.2]{Ku-DG}. An earlier version of the present paper was able to prove the existence of $X$ only for rings in which $2$ is a unit. The present version of the paper relies on the divided powers in $B$, $M$ and $K$ to avoid that hypothesis.

\begin{definition}
\label{5-21.1}
Adopt the notation of {\rm \ref{data17}} and  {\rm\ref{17.42}}.
\begin{enumerate}[\rm(a)]\item Define
$$B:\quad 0\to B_4\xrightarrow{b_4} B_3\xrightarrow{b_3} B_2\xrightarrow{b_2} B_1\xrightarrow{b_1}B_0$$ to be the modules
\begin{align*}
B_4&=D_2M_2, &&& B_3&=M_1\t M_2, &&\ts B_2=\bw^2M_1\p M_2,\\
B_1&=M_1, &&& B_0&=M_0,
\end{align*} and the maps
\begin{align}
b_4(\theta_2^{(2)})&{}= m_2(\theta_2)\t \theta_2\label{5-21.1a}\\
b_3(\theta_1\t \theta_2)&{}= \bmatrix -\theta_1\w m_2(\theta_2)\\m_1(\theta_1)\cdot \theta_2\endbmatrix\notag\\
b_2\left(\bmatrix \theta_1\w\theta_1'\\\theta_2\endbmatrix \right)&{}=m_1(\theta_1)\cdot \theta_1'-m_1(\theta_1')\cdot \theta_1+m_2(\theta_2), \text{ and}\notag\\
b_1&={}m_1.\notag\end{align}
\item Define $c_i:B_i\to K_i$ by
 \begin{align}c_4(\theta_2^{(2)})&{}=\beta_0(1)\cdot \beta_4(\theta_2^{(2)})-  \Big(\beta_2(\theta_2)\Big)^{(2)},\label{4.1.2}\\
c_3(\theta_1\t \theta_2)&{}=
\beta_0(1)\cdot \beta_3(\theta_1\cdot \theta_2)-\beta_1(\theta_1)\w \beta_2(\theta_2),\notag\\ 
c_2\left(\bmatrix \theta_1\w \theta_1'\\\theta_2\endbmatrix\right)&{} =\beta_0(1)\cdot \beta_2(\theta_1\theta_1')-\beta_1(\theta_1)\w\beta_1(\theta_1').\notag\end{align}
The maps $c_1$ and $c_0$ are both identically zero.
\end{enumerate}\end{definition}
Notice that the divided powers on the left side of (\ref{4.1.2}) take place in the formal divided power algebra $D_\bullet M_2$; the first divided power on the right side takes place in the $\DGGamma$-algebra $M$; and the second   divided power on the right side take place in the $\DGGamma$-algebra $K$.
\begin{observation}
\label{5-21.2}Retain the data of Definition~{\rm\ref{5-21.1}}. The following statements hold.
\begin{enumerate}[\rm(a)]
\item\label{5-21.2a} The maps and modules of $B$ form a complex.
\item\label{5-21.2b} The maps $c:B\to K$ form a map of complexes.
\item\label{5-21.2c} There are homotopy maps $h_i:B_i\to K_{i+1}$, for $0\le i\le 4$, such that \begin{enumerate}[\rm (i)]
\item $h_0$, $h_1$, and $h_4$ all are zero,
\item\label{5-21.2cii} the restriction of $h_2$ to the summand $M_2$ of $B_2$ is identically zero, and
\item\label{5-21.2ciii} $c_i=h_{i-1}\circ b_i+k_{i+1}\circ h_i$, for $1\le i\le 4$.
\end{enumerate}
\end{enumerate}
\end{observation}
\begin{proof} 
Assertion~(\ref{5-21.2a}) is obvious. We first prove (\ref{5-21.2b}).
Observe that
\begingroup\allowdisplaybreaks
\begin{align*}&(k_4\circ c_4)(\theta_2^{(2)})=k_4
\Big(
\beta_0(1)\cdot\beta_4(\theta_2^{(2)})-  \Big(\beta_2(\theta_2)\Big)^{(2)}\Big)
\\
{}={}&
\beta_0(1)\cdot (\beta_3\circ m_4)(\theta_2^{(2)})-
(k_2\circ \beta_2)(\theta_2)\w \beta_2(\theta_2),&&\text{by (\ref{17.6}) and (\ref{2.3.1}),}\\
{}={}&
\beta_0(1)\cdot \beta_3(m_2(\theta_2)\cdot \theta_2)-
(\beta_1\circ m_2)(\theta_2)\w \beta_2(\theta_2),&&\text{by (\ref{2.3.1}) and (\ref{17.6}),}\\ 
{}={}&c_3(m_2(\theta_2)\t \theta_2)=(c_3\circ b_4)(\theta_2^{(2)}).\end{align*}
\endgroup
Observe also that
\begingroup\allowdisplaybreaks
\begin{align*}&
(k_3\circ c_3)(\theta_1\t \theta_2)\\{}={}&k_3\Big(\beta_0(1)\cdot \beta_3(\theta_1\cdot \theta_2)-\beta_1(\theta_1)\w \beta_2(\theta_2)\Big)
\\
{}={}&\begin{cases}\beta_0(1)\cdot (\beta_2\circ m_3)(\theta_1\cdot \theta_2)\\
-(k_1\circ\beta_1)(\theta_1)\w \beta_2(\theta_2)
+\beta_1(\theta_1)\w (k_2\circ \beta_2)(\theta_2),\end{cases}&&\text{by (\ref{17.6}) and \ref{2.4}.(\ref{2.4.b})},
\\
{}={}&\begin{cases}
\beta_0(1)\cdot \Big(
m_1(\theta_1)\cdot \beta_2(\theta_2) -\beta_2(\theta_1\cdot m_2(\theta_2) 
)\Big)\\
-
\beta_0(1)\cdot 
m_1(\theta_1)\cdot \beta_2(\theta_2)
+\beta_1(\theta_1)\w (\beta_1\circ m_2)(\theta_2),\end{cases}&&\text{by \ref{2.4}.(\ref{2.4.b}) and  (\ref{17.6})},
\\
{}={}& -\beta_0(1)\cdot \beta_2(\theta_1\cdot m_2(\theta_2) 
)
+\beta_1(\theta_1)\w (\beta_1\circ m_2)(\theta_2)
\\
{}={}&(c_2\circ b_3)(\theta_1\t \theta_2).
\end{align*}
\endgroup
Finally, observe that
\begingroup\allowdisplaybreaks
\begin{align*}&
(k_2\circ c_2)\left(\bmatrix \theta_1\w \theta_1'\\\theta_2\endbmatrix\right)\\
{}={}&k_2\big(\beta_0(1)\cdot \beta_2(\theta_1\theta_1')-\beta_1(\theta_1)\w \beta_1(\theta_1')\Big)\\
{}={}&\begin{cases}\phantom{+}\beta_0(1)\cdot (\beta_1\circ m_2)(\theta_1\theta_1')\\ -(k_1\circ \beta_1)(\theta_1)\cdot \beta_1(\theta_1')+\beta_1(\theta_1)\cdot (k_1\circ \beta_1)(\theta_1'),\end{cases}&&\text{by (\ref{17.6}) and \ref{2.4}.(\ref{2.4.b})},\\
{}={}&\begin{cases}\phantom{+}\beta_0(1)\cdot \Big(m_1(\theta_1)\cdot \beta_1(\theta_1')
-m_1(\theta_1')\cdot \beta_1(\theta_1)\Big)\\
 -\beta_0(1)\cdot m_1(\theta_1)\cdot \beta_1(\theta_1')+\beta_1(\theta_1)\cdot \beta_0(1)\cdot m_1(\theta_1'),\end{cases}
&&\text{by \ref{2.4}.(\ref{2.4.b}) and  (\ref{17.6})},
\\{}={}&0=c_1\circ b_2\left(\bmatrix \theta_1\w \theta_1'\\\theta_2\endbmatrix\right).\end{align*}\endgroup
This completes the proof of (\ref{5-21.2b}); now we prove (\ref{5-21.2c}).
The map $c:B\to K$ is a map of complexes with $c_0$ and $c_1$ both identically zero; furthermore, $K$ is a resolution. It follows that there is a homotopy $$\{h_i:B_i\to K_{i+1}\mid 0\le i\le 4\}$$ which satisfies condition (\ref{5-21.2ciii}). It is clear that $h_0$ and $h_1$ may be chosen to be zero. The target for $h_4$ is zero; so this map also is zero. The restriction of $h_2$ to $M_2$ may be taken to be any homomorphism which completes the homotopy  
$$\xymatrix{&&&M_2\ar[rrr]^{m_2}\ar[d]^(.4){c_2|_{M_2}=0} \ar[dlll]_(.4){h_2|_{M_2}}&&&M_1\ar[dlll]^(.4){h_1=0}\\ 
K_3\ar[rrr]^{k_3}&&&K_2,
}$$in the sense that $$c_2|_{M_2}=h_1\circ m_2+k_3\circ h_2|_{M_2}.$$
The maps $c_2|_{M_2}$ and $h_1$ are already identically zero. Consequently, one may choose $h_2|_{M_2}$ to be identically zero.
\end{proof}

\begin{lemma}
\label{17.40} Adopt the notation of {\rm \ref{data17}} and  {\rm\ref{17.42}}. Then there exists a homomorphism $X:M_1\to M_2$  which satisfies {\rm\ref{17.9}.(\ref{17.9.b})} and {\rm\ref{17.9}.(\ref{17.9.c})}.\end{lemma}

\begin{proof}
Let $\{h_i:B_i\to K_{i+1}\}$ be the homotopy of Observation~\ref{5-21.2}.(\ref{5-21.2c}).
Define $$X:M_1\to M_2$$ by 
\begin{equation}\label{5-20.1}
X(\theta_1)\cdot \theta_2=(\beta_4^{-1}\circ h_3)(\theta_1\t \theta_2).\end{equation} (Recall from \ref{19.3.7} that $\beta_4$ is an isomorphism.) We first  prove that $X$ satisfies  {\rm\ref{17.9}.(\ref{17.9.c})}. Observe that 
\begingroup\allowdisplaybreaks\begin{align*}&(X\circ m_2+m_3\circ X^\dagger)(\theta_2)\cdot \theta_2'\\{}={}&(\beta_4^{-1}\circ h_3)(m_2(\theta_2)\t \theta_2')+\theta_2\cdot X(m_2(\theta_2')),&&\text{by (\ref{17.5}) and \ref{17.8},}\\
{}={}&(\beta_4^{-1}\circ h_3)(m_2(\theta_2)\t \theta_2')+
(\beta_4^{-1}\circ h_3)(m_2(\theta_2')\t \theta_2)\\
{}={}&(\beta_4^{-1}\circ h_3)(b_4(\theta_2\theta_2'))\\
{}={}&(\beta_4^{-1}\circ c_4)(\theta_2\theta_2'),&&\text{by Obs.~\ref{5-21.2}.(\ref{5-21.2ciii})},\\
{}={}&\beta_0(1)\cdot \theta_2\theta_2'-\beta_4^{-1}\Big(\beta_2(\theta_2)\w \beta_2(\theta_2')\Big)\\
{}={}&\beta_0(1)\cdot \theta_2\theta_2'-\beta_4^{-1}\Big(\beta_4\big(\theta_2\cdot (\a_2\circ\beta_2)(\theta_2')\big)\Big),&&\text{by Obs.~\ref{17.3}.(\ref{17.3.c})},\\
{}={}&\beta_0(1)\cdot \theta_2\theta_2'-\theta_2\cdot (\a_2\circ \beta_2)(\theta_2')\\
{}={}&\beta_0(1)\cdot \theta_2\theta_2'-(\a_2\circ\beta_2)(\theta_2)\cdot \theta_2'&&\text{by Obs.~\ref{17.3}.(\ref{17.3.b})},\\
{}={}&
\Big(\beta_0(1)\cdot \id_{M_2}-\a_2\circ \beta_2\Big)(\theta_2)\cdot \theta_2'.
\end{align*} \endgroup
Now we prove that $X$ satisfies {\rm\ref{17.9}.(\ref{17.9.b})}. 
Recall from (\ref{17.5}) that 
\begin{equation}\label{4.3.2}(m_2\circ X)(\theta_1)\cdot \theta_3=-X(\theta_1)\cdot m_3(\theta_3)=
-(\beta_4^{-1}\circ h_3)(\theta_1\t m_3(\theta_3)).\end{equation}
Apply Observation~\ref{5-21.2}.(\ref{5-21.2ciii}) to see that
$$k_4\circ h_3+h_2\circ b_3=c_3.$$ Observe that  $$b_3(\theta_1\t m_3(\theta_3))=\bmatrix 0\\m_1(\theta_1)\cdot m_3(\theta_3)\endbmatrix,$$which is in the summand $M_2$ of $T_2$. It follows from Obs.~\ref{5-21.2}.(\ref{5-21.2cii}) that $$(h_2\circ b_3)(\theta_1\t m_3(\theta_3))=0;$$ and therefore, 
\begin{align*}&(k_4\circ h_3)(\theta_1\t m_3(\theta_3))=c_3(\theta_1\t m_3(\theta_3))\\
{}={}&\beta_0(1)\cdot \beta_3((\theta_1\cdot m_3(\theta_3))-\beta_1(\theta_1)\w (\beta_2\circ m_3)(\theta_3).\intertext{Use the commutative diagram (\ref{17.6}) to write $\beta_2\circ m_3$ as $k_3\circ \beta_3$ and then use the product rule \ref{2.4}.(\ref{2.4.b}) on each summand. It follows that $(k_4\circ h_3)(\theta_1\t m_3(\theta_3))$ is equal to}
{}={}&\begin{cases}\phantom{+}\beta_0(1)\cdot \Big(m_1(\theta_1)\cdot \beta_3(\theta_3)-(\beta_3\circ m_4)(\theta_1\cdot \theta_3)\Big)\\- \Big(\beta_0(1)\cdot m_1(\theta_1)\cdot \beta_3(\theta_3)-k_4\big(\beta_1(\theta_1)\w \beta_3(\theta_3)\big)\Big)\end{cases}\\
{}={}&k_4\Big(-\beta_0(1)\cdot\beta_4(\theta_1\cdot \theta_3) +\beta_1(\theta_1)\w \beta_3(\theta_3) \Big),&&\text{by (\ref{17.6}).}\end{align*}
The map $k_4$ is injective; hence, 
$$h_3(\theta_1\t m_3(\theta_3))=-\beta_0(1)\cdot\beta_4(\theta_1\cdot \theta_3) +\beta_1(\theta_1)\w \beta_3(\theta_3),$$
and (\ref{4.3.2}) now becomes
\begin{align*}(m_2\circ X)(\theta_1)\cdot \theta_3={}&-(\beta_4^{-1}\circ h_3)(\theta_1\t m_3(\theta_3))\\
{}={}&\beta_4^{-1}\Big(\beta_0(1)\cdot\beta_4(\theta_1\cdot \theta_3) -\beta_1(\theta_1)\w \beta_3(\theta_3)\Big).\end{align*}
Recall  from assertions (\ref{17.3.c}) and (\ref{17.3.b}) of Observation~\ref{17.3}
that \begin{align*}\beta_4^{-1}\Big(\beta_1(\theta_1)\w \beta_3(\theta_3)\Big)
={}&\beta_4^{-1}\Big(\beta_4\big(\theta_1\cdot (\a_3\circ\beta_3)(\theta_3)\big)\Big)=\theta_1\cdot (\a_3\circ\beta_3)(\theta_3)\\
{}={}&(\a_1\circ \beta_1)(\theta_1)\cdot \theta_3.\end{align*}
Thus, $(m_2\circ X)(\theta_1)\cdot \theta_3=\beta_0(1)\cdot \theta_1\cdot \theta_3
-(\a_1\circ \beta_1)(\theta_1)\cdot \theta_3$ and $X$ satisfies \ref{17.9}.(\ref{17.9.b}).
\end{proof}

\section{There exists a homomorphism $X$ which satisfies {\rm\ref{17.9}.(\ref{17.9.a})}, {\rm\ref{17.9}.(\ref{17.9.b})}, and {\rm\ref{17.9}.(\ref{17.9.c})}.}\label{6}

Lemma~\ref{17.34} is the main result in this section; its proof
 is given in \ref{17.32}.

\begin{lemma}
\label{17.34} Adopt the notation of {\rm \ref{data17}} and  {\rm\ref{17.42}}.
 Let
 $X:M_1\to M_2$ be the  homomorphism of Lemma~{\rm\ref{17.40}}. Then there is a homomorphism $U:M_1\to M_3$ such that $X'=X-m_3\circ U$
 satisfies {\rm\ref{17.9}.(\ref{17.9.a})}, 
{\rm\ref{17.9}.(\ref{17.9.b})}, and {\rm\ref{17.9}.(\ref{17.9.c})}.
\end{lemma}

The map $u$ of Observation~\ref{17.28} is a first approximation of the map $U$ which is promised in Lemma~\ref{17.34}. The map $u$ will be modified in Lemma~\ref{17.29} and Definition~\ref{17.35}. 

\begin{observation}
\label{17.28}  Adopt the notation of {\rm \ref{data17}} and  {\rm\ref{17.42}}.
 Let
 $X:M_1\to M_2$ be the  homomorphism of Lemma~{\rm\ref{17.40}}. Then there exists a homomorphism $u:K_1\to M_3$ such that
\begin{equation}\notag X\circ \a_1=m_3\circ u: K_1\to M_3.\end{equation}
\end{observation}

\begin{proof}
Consider $\a_1$ followed by  \ref{17.9}.(\ref{17.9.b}):
$$m_2\circ (X\circ \a_1)= \big(\beta_0(1)\cdot \id_{M_1}-\alpha_1\circ\beta_1\big)\circ \a_1.$$ Apply Observation~\ref{17.3}.(\ref{17.3.a}) to see that the right side of the previous equation is zero. It follows that $m_2\circ (X\circ \a_1)$ is identically zero. The complex $\Hom_P(K_1,M)$ is acyclic; hence there exists a homomorphism $u:K_1\to M_3$ such that
\begin{equation}\notag X\circ \a_1=m_3\circ u: K_1\to M_3.\end{equation}
\vskip-24pt\end{proof}

In order to  modify $u$ (and $X$), we use the homotopy of Observation~\ref{5-21.2}.(\ref{5-21.2c}). The homotopy map $h_3$   gave rise to the homomorphism $X:M_1\to M_2$ of Lemma~\ref{17.40}; but $h_3$ contains information about $X$ that we have  not yet exploited. 
\begin{observation} 
\label{5-20.5} The restriction of the map $c_4$ of Definition~{\rm\ref{5-21.1}} to $D_2(\im \a_2)$
 is identically zero.\end{observation}
\begin{proof}
The map 
 $k_4$ is an injection. It suffices to show that \begin{equation}\label{5-20.4}(k_4\circ c_4)\Big((\a_2(\phi_2))^{(2)}\Big)=0,\end{equation}for each $\phi_2\in K_2$. 
The left side of (\ref{5-20.4}) is 
$$\beta_0(1)\cdot (k_4\circ \beta_4)\Big(\big(\a_2(\phi_2)\big)^{(2)}\Big)-k_4\Big(\big((\beta_2\circ \a_2)(\phi_2)\big)^{(2)}\Big).$$
Apply (\ref{17.6}), (\ref{2.3.1}),  assertions (\ref{17.3.c}) and (\ref{17.3.a}) of Observation~\ref{17.3}, and the Divided Power axiom (\ref{2.2.a})
  to see that
\begin{align*}&\beta_0(1)\cdot (k_4\circ \beta_4)\Big(\big(\a_2(\phi_2)\big)^{(2)}\Big)
= \beta_0(1)\cdot (\beta_3\circ m_4)\Big(\big(\a_2(\phi_2)\big)^{(2)}\Big)\\
{}={}& \beta_0(1)\cdot \beta_3\Big(
(m_2\circ \a_2)(\phi_2)\cdot \a_2(\phi_2)\Big)
= \beta_0(1)\cdot \beta_3\Big(
(\a_1\circ k_2)(\phi_2)\cdot \a_2(\phi_2)\Big)\\
{}={}& \beta_0(1)\cdot 
(\beta_1\circ \a_1\circ k_2)(\phi_2)\wedge \phi_2
=\beta_0(1)^2\cdot k_2(\phi_2)\wedge \phi_2\\
{}={}& \beta_0(1)^2\cdot k_4\big(\phi_2^{(2)}\big)
=k_4\Big(\big(\beta_0(1)(\phi_2)\big)^{(2)}\Big)\\
{}={}&k_4\Big(\big((\beta_2\circ \a_2)(\phi_2)\big)^{(2)}\Big).
\end{align*}
Thus, (\ref{5-20.4}) is established and the proof is complete.
\end{proof}

The serious work in this section is done in the proof of Lemma~\ref{17.29}.
\begin{lemma}
\label{17.29} Adopt the notation of {\rm \ref{data17}} and  {\rm\ref{17.42}} and let $u$ be the homomorphism of Observation~{\rm\ref{17.28}}. Then there exists  a homomorphism $v: K_1\to M_4$ such that 
the homomorphism 
\begin{equation}\label{17.37}u'=(u+m_4\circ v):K_1\to M_3,\end{equation} 
satisfies
$$u'(\phi_1)\cdot \a_1(\phi_1')+u'(\phi_1')\cdot \a_1(\phi_1)=0\quad\text{and}\quad u'(\phi_1)\cdot \a_1(\phi_1)=0,$$ for all $\phi_1,\phi_1'$ in $K_1$. 
\end{lemma}

\begin{proof}
Let $\phi_1$ and $\phi_1'$ be elements of $K_1$. Consider the element $\Big(\a_2(\phi_1\w \phi_1')\Big)^{(2)}$ of $D_2M_2$. Observe that 
\begingroup\allowdisplaybreaks
\begin{align*}0={}&(\beta_4^{-1}\circ c_4)\Big(\Big(\a_2(\phi_1\w \phi_1')\Big)^{(2)}\Big), &&\text{by Obs.~\ref{5-20.5},}\\
{}={}&(\beta_4^{-1}\circ h_3\circ b_4)\Big(\Big(\a_2(\phi_1\w \phi_1')\Big)^{(2)}\Big),&&\text{by Obs.~\ref{5-21.2}.(\ref{5-21.2ciii}),}\\
{}={}&(\beta_4^{-1}\circ h_3)\Big((m_2\circ \a_2)(\phi_1\w \phi_1')\t \a_2(\phi_1\w \phi_1')\Big),&&\text{by (\ref{5-21.1a}),}\\
{}={}&X\Big((m_2\circ \a_2)(\phi_1\w \phi_1')\Big)\cdot \a_2(\phi_1\w \phi_1'),&&\text{by (\ref{5-20.1}),}\\
{}={}&(X\circ \a_1)\Big(k_2(\phi_1\w \phi_1')\Big)\cdot \a_2(\phi_1\w \phi_1'),&&\text{by (\ref{17.6}),}\\
{}={}&\big((m_3\circ u)(k_2(\phi_1\w \phi_1'))\big)\cdot(\a_2(\phi_1\w \phi_1')),&&\text{by Obs.~\ref{17.28},}\\
{}={}&\big(u(k_2(\phi_1\w \phi_1'))\big)\cdot((m_2\circ\a_2)(\phi_1\w \phi_1')),&&\text{by (\ref{17.5}),}\\
{}={}&\big(u(k_2(\phi_1\w \phi_1'))\big)\cdot((\a_1\circ k_2)(\phi_1\w \phi_1')),&&\text{by (\ref{17.6}).}
\end{align*}\endgroup
The differential in the Koszul complex yields
$$0=\big(k_1(\phi_1)\cdot u(\phi_1')-k_1(\phi_1')\cdot u(\phi_1)
\big)\cdot\big(k_1(\phi_1)\cdot \a_1(\phi_1')-k_1(\phi_1')\cdot \a_1(\phi_1)\big);$$
hence,
\begin{equation}
0=\begin{cases} 
\phantom{+}k_1(\phi_1)\cdot k_1(\phi_1)\cdot\big(u(\phi_1')\cdot\a_1(\phi_1')\big)\\
-k_1(\phi_1')\cdot k_1(\phi_1)\cdot\big(u(\phi_1)\cdot\a_1(\phi_1')+u(\phi_1')\cdot\a_1(\phi_1)\big)\\
+k_1(\phi_1')\cdot k_1(\phi_1')\cdot\big(u(\phi_1)\cdot\a_1(\phi_1)\big).\\
\end{cases}\label{17.30}\end{equation}
for all $\phi_1$ and $\phi_1'$ in $K_1$.

Thus,
$$(k_1(\phi_1))^2 \cdot \big(u(\phi_1')\cdot \a_1(\phi_1')\big)\subseteq (k_1(\phi_1'))M_4$$ for all
$\phi_1$ and $\phi_1'$ in $K_1$. The image of $k_1$ is an ideal in $P$ of grade $4$ and $M_4$ is isomorphic to $P$. Assume that $\phi_1'$ is an element of $K_1$ with $k_1(\phi_1')$ a regular element in $P$. In this case every associated prime of $P/(k_1(\phi_1'))$ has grade one; and therefore, $$u(\phi_1')\cdot \a_1(\phi_1')\in (k_1(\phi_1'))M_4.$$ 
Lemma~\ref{apr29-19}
guarantees that $K_1$ has a basis $\phi_{1,1}, \phi_{1,2},\phi_{1,3},\phi_{1,4}$ with the property that $k_1(\phi_{1,i})$ is a regular element of $P$ for each $i$. For each $i$, we identify
an element $v(\phi_{1,i})\in M_4$ with 
\begin{equation}\label{17.38}u(\phi_{1,i})\cdot \a_1(\phi_{1,i})= k_1(\phi_{1,i})\cdot v(\phi_{1,i}).\end{equation}
Extend $v$ to be a homomorphism $v:K_1\to M_4$. 
Take $a$ and $b$ from the set $$\{\phi_{1,1}, \phi_{1,2},\phi_{1,3}, \phi_{1,4}\}$$ and 
rewrite (\ref{17.30}) as
$$0=\begin{cases} 
+k_1(a)k_1(a)\cdot\big(k_1(b)\cdot v(b)\big)\\
-k_1(b)k_1(a)\cdot\big(u(a)\cdot\a_1(b)+u(b)\cdot\a_1(a)\big)\\
+k_1(b)k_1(b)\cdot\big(k_1(a)\cdot v(a)\big)\\
\end{cases}$$
Use the fact that $k_1(a)$ and $k_1(b)$ are regular elements of $P$ in order to  
see that
$$0=
k_1(a) \cdot v(b)
-\big(u(a)\cdot\a_1(b)+u(b)\cdot\a_1(a)\big)
+k_1(b)\cdot v(a).
$$
In other words,
\begin{align*}
u(a)\cdot\a_1(b)+u(b)\cdot\a_1(a)&=k_1(a) \cdot v(b)+k_1(b)\cdot v(a)\\
&=-m_4(v(b))\cdot \a_1(a)-m_4(v(a))\cdot \a_1(b),
\end{align*}
for all $a,b\in \{\phi_{1,1}, \phi_{1,2},\phi_{1,3},\phi_{1,4}\}$. (The last equality uses the product rule of (\ref{17.5}) and the equality $m_1\circ \a_1=k_1$ of the Commutative Diagram (\ref{17.6}).) Similarly, we deduce directly  from (\ref{17.38}) that
$$u(a)\cdot \a_1(a)=-m_4(v(a))\cdot \a_1(a),$$ for $a\in  \{\phi_{1,1}, \phi_{1,2},\phi_{1,3},\phi_{1,4}\}$. Thus,
\begin{equation}\label{17.39}\begin{array}{rl}
(u+m_4\circ v)(a)\cdot \a_1(a)=0\phantom{,}&\text{and}\\
(u+m_4\circ v)(a)\cdot \a_1(b)+(u+m_4\circ v)(b)\cdot \a_1(a)=0,\end{array}\end{equation}for all $a,b\in \{\phi_{1,1}, \phi_{1,2},\phi_{1,3},\phi_{1,4}\}$. It now follows that (\ref{17.39}) holds for all $a$ and $b$ in 
$K_1$. 
\end{proof}

\begin{definition}
\label{17.35}
 Adopt the notation of {\rm \ref{data17}} and  {\rm\ref{17.42}} and let $u'$ be the homomorphism of {\rm(\ref{17.37})}. 
Define \begin{equation}\notag U:M_1\to M_3\end{equation} by
$$\begin{cases}U(\a_1(\phi_1))=u'(\phi_1),&\text{if $\phi_1\in K_1$,}\\
U(\theta_{1,2})\cdot \a_1(\phi_1) = -u'(\phi_1)\cdot \theta_{1,2},&\text{if $\theta_{1,2}\in M_{1,2}$ and $\phi_1\in K_1$, and}\\
U(\theta_{1,2})\cdot M_{1,2}=0&\text{if $\theta_{1,2}\in M_{1,2}$.}\end{cases}$$\end{definition}

\begin{remarks}
\begin{enumerate}[\rm(a)]
\item Notice that \begin{equation}\label{17.33}U(\theta_1)\cdot \theta_1'+U(\theta_1')\cdot \theta_1=0,\end{equation}
for all $\theta_1,\theta_1'\in M_1$.
\item Recall from  {\rm\ref{17.42}}  that $M_1=M_{1,1}\p M_{1,2}$ and $\a_1:K_1\to M_{1,1}$ is an isomorphism. It follows that $U$ is a well-defined homomorphism on all of $M_1$.
\end{enumerate}
\end{remarks}

\begin{chunk}
\label{17.32} {\it Proof of Lemma~{\rm\ref{17.34}}.}
Let $U$ be the homomorphism of Definition~\ref{17.35} and let
 \begin{equation}\label{17.36}X'=X-m_3\circ U.\end{equation}
We prove that
the homomorphism  $X'$ of {\rm(\ref{17.36})} satisfies hypotheses {\rm\ref{17.9}.(\ref{17.9.a})}, {\rm\ref{17.9}.(\ref{17.9.b})}, and {\rm\ref{17.9}.(\ref{17.9.c})}.
 Hypothesis \ref{17.9}.(\ref{17.9.b}) holds because $$m_2\circ X'=
m_2\circ (X-m_3\circ U)=
m_2\circ X.$$
Hypothesis \ref{17.9}.(\ref{17.9.c}) holds because
\begin{align*}&\Big(m_3\circ U\circ m_2+m_3\circ (m_3\circ U)^\dagger\Big)(\theta_2)\cdot \theta_2'\\
{}={}&U(m_2\theta_2)\cdot m_2(\theta_2')+U(m_2\theta_2')\cdot m_2(\theta_2)=0
.\end{align*}(The first equality uses (\ref{17.5}), (\ref{17.8}), and the graded-commutativity of $M$; the second uses (\ref{17.33}).) It follows that  
\begin{align*}X'\circ m_2+m_3\circ (X')^\dagger&{}= 
(X-m_3\circ U)\circ m_2+m_3\circ (X-m_3\circ U)^\dagger\\&{}=
X\circ m_2+m_3\circ X^\dagger.\end{align*}
Hypothesis \ref{17.9}.(\ref{17.9.a}) holds because
 \begin{align*}X'\circ \a_1&{}=X\circ \a_1-m_3\circ U\circ \a_1,&&\text{by (\ref{17.36}),}
\\&{}=X\circ \a_1-m_3u',&&\text{by Def.~\ref{17.35},}
\\&{}=X\circ \a_1-m_3\circ(u+ m_4\circ v),&&\text{by (\ref{17.37}),}
\\&{}=X\circ \a_1-m_3\circ u,&&\text{because $M$ is a complex,}
\\&{}=0,&&\text{by Obs.~\ref{17.28}.}\end{align*}
This completes the proof 
 of Lemma~\ref{17.34}.  \hfil \qed
\end{chunk}

\section{The map $X$ of Theorem~\ref{17.9} exists.}\label{7}

In Lemma~\ref{17.34} we produced a map $X:M_1\to M_2$ which satisfies hypotheses \ref{17.9}.(\ref{17.9.a}), \ref{17.9}.(\ref{17.9.b}), and \ref{17.9}.(\ref{17.9.c}). In Lemma~\ref{17.43} we show the $X$ also satisfies 
\ref{17.9}.(\ref{17.9.d}), and \ref{17.9}.(\ref{17.9.e}). No further modification is needed.

\begin{lemma}
\label{17.15} Adopt the notation of {\rm \ref{data17}} and  {\rm\ref{17.42}}. If $X:M_1\to M_2$ is a homomorphism which satisfies {\rm\ref{17.9}.(\ref{17.9.a})}, then $$\ker m_3 \cap \im X^\dagger =0.$$\end{lemma}

\begin{proof}The complex $M$ is acyclic; so it suffices to prove that $$\im m_4\cap \im X^\dagger =0.$$ Suppose that $\theta_4\in M_4$ and $\theta_2\in M_2$ with
\begin{equation}\label{17.16}m_4(\theta_4)=X^\dagger (\theta_2).\end{equation} Let $\phi_1$ be an arbitrary element of $K_1$. Observe that
\begin{align*}-k_1(\phi_1)\cdot \theta_4{}={}&-m_1(\a_1(\phi_1))\cdot \theta_4,&&\text{by   (\ref{17.6}),}\\
{}={}&m_4(\theta_4)\cdot \a_1(\phi_1),&&\text{by (\ref{17.5}),}\\
{}={}& X^\dagger (\theta_2)\cdot \a_1(\phi_1),&&\text{by (\ref{17.16}),}\\{}={}&
\theta_2\cdot (X\circ\a_1)(\phi_1),&&\text{by (\ref{17.8}),}\\
{}={}&0,&&\text{by \ref{17.9}.(\ref{17.9.a}).}\end{align*}
Thus, the ideal $\im k_1$, which has positive grade, annihilates the element of $\theta_4$ of $M_4$. Recall that the module $M_4$ is isomorphic to $P$. It follows that $\theta_4$ is zero.
\end{proof}

\begin{lemma}
\label{17.43}If $X:M_1\to M_2$ is a homomorphism which satisfies {\rm\ref{17.9}.(\ref{17.9.a})}, 
{\rm\ref{17.9}.(\ref{17.9.b})}, and {\rm\ref{17.9}.(\ref{17.9.c})}, then $X$ also satisfies
{\rm\ref{17.9}.(\ref{17.9.d})} and {\rm\ref{17.9}.(\ref{17.9.e})}.
\end{lemma}
\begin{proof} We first prove \ref{17.9}.(\ref{17.9.e}). 
Consider $\a_2$ followed by 
 \ref{17.9}.(\ref{17.9.c}): 
\begin{equation}\label{17.17}X\circ m_2\circ\a_2+ m_3\circ X^\dagger\circ \a_2=(\beta_0(1)\cdot \id_{M_2}-\a_2\circ \beta_2)\circ \a_2.\end{equation} The right side of (\ref{17.17})
is zero by Observation~\ref{17.3}.(\ref{17.3.a}); the composition ${X\circ m_2\circ\a_2}$ is equal to 
$$X\circ \a_1\circ k_2=0$$ by (\ref{17.6}) and \ref{17.9}.(\ref{17.9.a}). Thus, 
equation (\ref{17.17}) yields 
$$m_3\circ X^\dagger\circ \a_2=0;$$ and therefore, $$\im (X^\dagger\circ \a_2)\subseteq \ker m_3 \cap \im X^\dagger=0$$ by Lemma~\ref{17.15}. This establishes \ref{17.9}.(\ref{17.9.e}).

Now we prove \ref{17.9}.(\ref{17.9.d}). Consider $X$ followed by 
 \ref{17.9}.(\ref{17.9.c}):
\begin{equation}\label{17.18}X\circ m_2\circ X+ m_3\circ X^\dagger\circ X=\beta_0(1)\cdot \id_{M_2}\circ X-\a_2\circ \beta_2\circ X.\end{equation}
Observe  that $$\alpha_2\circ\beta_2\circ X=0.$$Indeed,
\begin{align*}&[\theta_2\cdot (\alpha_2\circ\beta_2)(X(\theta_1))]_M\\
{}={}& [\beta_2(\theta_2)\wedge \beta_2(X(\theta_1))]_K,&&\text{by (\ref{17.2})},\\
{}={}& [\beta_2(X(\theta_1))\wedge \beta_2(\theta_2)]_K,&&\text{because $K$ is graded-commutative,}\\
{}={}& [X(\theta_1)\cdot (\alpha_2\circ \beta_2)(\theta_2)]_M,&&\text{by (\ref{17.2})},\\
{}={}& [(\alpha_2\circ \beta_2)(\theta_2)\cdot X(\theta_1)]_M,&&\text{because $M$ is graded-commutative,}\\
{}={}&[(X^\dagger \circ\alpha_2\circ \beta_2)(\theta_2)\cdot \theta_1]_M,&&\text{by (\ref{17.8}),}\\
{}={}&0,&&\text{by \ref{17.9}.(\ref{17.9.e})}.\end{align*}
Apply {\rm\ref{17.9}.(\ref{17.9.b})} and  {\rm\ref{17.9}.(\ref{17.9.a})} to see that
$$X\circ m_2\circ X=X\circ (\beta_0(1)\cdot \id_{M_1}-\alpha_1\circ \beta_1)=\beta_0(1)\cdot X.$$
Thus, equation (\ref{17.18}) is 
$$\beta_0(1)\cdot X+ m_3\circ X^\dagger\circ X=\beta_0(1)\cdot X$$or $m_3\circ X^\dagger\circ X=0$. 
It follows that $$\im (X^\dagger\circ X)\subseteq \ker m_3\cap \im X^\dagger=0$$ by Lemma~\ref{17.15}. This establishes \ref{17.9}.(\ref{17.9.d}).
\end{proof}

\begin{theorem}
\label{17.11} Adopt the language of {\rm\ref{17.42}}. Then there exists a map $$X:M_1\to M_2$$ such that the hypotheses of Theorem~{\rm\ref{17.9}} hold.
\end{theorem}
\begin{proof}
Apply Lemma~\ref{17.34}
followed by Lemma~\ref{17.43}. \end{proof}

\section{Further properties of $X$.}\label{PC2}

We continue Section~\ref{PC}. Now that we have proven that the map $X$ of Theorem~\ref{17.9} exists, 
 we deduce further properties of $X$. 
These formulas, together with those of Section~\ref{PC},
provide
the   proofs of 
Theorems~\ref{17.9} and \ref{thm17}. There are many of these formulas; but each proof is straightforward. 

\begin{observation}\label{17*.2} The map $X$ of Theorem~{\rm\ref{17.9}} satisfies the following identities{\rm:}

\begin{enumerate} [\rm(a)]\item\label{17*.2.a} $\beta_3\circ X^\dagger=0$,

\item\label{17*.2.b} $\beta_2\circ X=0$,

\item\label{17*.2.c}$w_3\circ X^\dagger=0$,

\item\label{18.H} 
$\im X^\dagger\subseteq M_{3,2}$\,,  and

\item\label{17*.2.d} $X^\dagger\circ m_3+\a_3\circ \beta_3=\beta_0(1)\cdot \id_{M_3}$.
\end{enumerate}\end{observation}

\begin{proof} (\ref{17*.2.a}) Use (\ref{17.2}), (\ref{17.8}), and Hypothesis \ref{17.9}.(\ref{17.9.a}) to see that  $$[(\beta_3\circ X^\dagger)(\theta_2)\wedge \phi_1]_K=[X^\dagger(\theta_2)\cdot \a_1(\phi_1)]_M=[\theta_2\cdot (X\circ \a_1)(\phi_1)]_M=0.$$

\ms\noindent(\ref{17*.2.b})  Use (\ref{17.2}), the graded-commutativity of $M$, (\ref{17.8}), and Hypothesis \ref{17.9}.(\ref{17.9.e}) to see that 

\begin{align*}[(\beta_2\circ X)(\theta_1)\w \phi_2]_K&{}=[X(\theta_1)\cdot \a_2(\phi_2)]_M
=[\a_2(\phi_2)\cdot X(\theta_1)]_M\\
&{}=[X^\dagger(\a_2(\phi_2))\cdot \theta_1]_M=0.\end{align*}

\ms\noindent(\ref{17*.2.c}) Apply Definition~\ref{17.8'}, (\ref{17.8}), and Hypothesis \ref{17.9}.(\ref{17.9.a}) to see that $$(w_3\circ X^\dagger)(\theta_2)=(X^\dagger(\theta_2))\cdot \a_1(\sigma)=\theta_2\cdot X(\a_1(\sigma))=0.$$

\ms\noindent(\ref{18.H}) It suffices to show that $X^\dagger(\theta_2)\cdot \a_1(\phi_1)=0$ and this is obvious from the definition of $^\dagger$ and \ref{17.9}.(\ref{17.9.a}) as shown in the proof of (\ref{17*.2.c}).

\ms\noindent(\ref{17*.2.d}) Observe that \begin{align*}
(X^\dagger\circ m_3)(\theta_3)\cdot \theta_1&{}=m_3(\theta_3)\cdot X(\theta_1),
&&\text{by (\ref{17.8}),}
\\&{}=\theta_3\cdot (m_2\circ X)(\theta_1),&&\text{by (\ref{17.5}),}\\
&{}=\theta_3\cdot\big(\beta_0(1)\cdot \theta_1-(\a_1\circ\beta_1)(\theta_1)\big),&&\text{by \ref{17.9}.(\ref{17.9.b}),}
\\&{}=\big(\beta_0(1)\cdot \id_{M_3}-\a_3\circ\beta_3\big)(\theta_3)\cdot \theta_1,
&&\text{by \ref{17.3}.(\ref{17.3.b}).}
\end{align*}
\end{proof}

\begin{lemma}\label{17*.5'} In the language of Definition~{\rm\ref{17.42}} and Theorem~{\rm\ref{17.9}}, the following identities hold{\rm:}
\begin{enumerate}[\rm(a)]
\item\label{17*.5'.a} $w_2\circ X=X^\dagger\circ w_1$,
\item\label{17*.5'.b} $Y\circ w_1|_{M_{1,2}}=0$, and 
\item\label{17*.5'.c} $(\beta_2\circ w_1+Y\circ X)|_{M_{1,2}}=0$,  
\item\label{18.K}
$w_1\circ \proj_{M_{1,2}}\circ m_2+\a_2\circ Y=w_1\circ m_2$, 
\item\label{18.J} $W\circ \beta_2+m_3\circ \proj_{M_{3,2}}\circ w_2=m_3\circ w_2$, 
\item\label{18.O}
$\proj_{M_{3,2}}\circ w_2\circ W=0$,   
\item\label{18.P} $\proj_{M_{3,2}}\circ(X^\dagger\circ W+ w_2\circ \a_2 )=0$, and
\item\label{18.M}$\beta_2\circ W-Y\circ \a_2=k_1(\sigma)\cdot \id_{K_2}$.
\end{enumerate}
\end{lemma}
\begin{proof} (\ref{17*.5'.a}) We prove $$\im (w_2\circ X-X^\dagger\circ w_1)\subseteq (\ker m_3\cap \ker \beta_3)$$and then apply Lemma~\ref{17*.5}. Observe that 
$\beta_3\circ X^\dagger=0$ by Observation \ref{17*.2}.(\ref{17*.2.a}) and $$\beta_3\circ w_2\circ X=z_2\circ \beta_2\circ X=0$$ by Observation~\ref{17*.4} and Observation \ref{17*.2}.(\ref{17*.2.b}). It follows that  $$\im (w_2\circ X-X^\dagger\circ w_1)\subseteq  \ker \beta_3.$$
We complete the proof by showing that $\im (w_2\circ X-X^\dagger\circ w_1)\subseteq  \ker m_3$.
Observe that 
\begingroup\allowdisplaybreaks\begin{align*}
&m_3\circ(w_2\circ X-X^\dagger\circ w_1)\\{}={}&w_1\circ m_2\circ X
+(k_1(\sigma))\cdot X
-m_3\circ X^\dagger\circ w_1,
&&\text{by \ref{17*.3}.(\ref{17*.3.b}),}\\
{}={}&w_1\circ (\beta_0(1)\cdot \id_{M_1}-\a_1\circ \beta_1)+(k_1(\sigma))\cdot X-m_3\circ X^\dagger\circ w_1,
&&\text{by \ref{17.9}.(\ref{17.9.b}),}\\\intertext{(Use Observation \ref{17*.4} twice to see that 
$w_1\circ \a_1\circ \beta_1=\a_2\circ z_1\circ \beta_1=\a_2\circ \beta_2\circ w_1$.)}
{}={}&\Big(\beta_0(1)\cdot\id_{M_2}-\a_2\circ \beta_2-m_3\circ X^\dagger\Big)\circ w_1+(k_1(\sigma))\cdot X\\
{}={}&X\circ m_2\circ w_1+(k_1(\sigma))\cdot X,&&\text{by \ref{17.9}.(\ref{17.9.c}),}\\
\intertext{(Use Observation \ref{17*.3}.(\ref{17*.3.b}), again, to see that $m_2\circ w_1=w_0\circ m_1- k_1(\sigma)\cdot \id_{M_1}$.)}{}={}&X\circ w_0\circ m_1=0,&&\text{by \ref{17.9}.(\ref{17.9.a}),}
\end{align*}\endgroup
since $w_0(1)=\a_1(\sigma)$.

\ms \noindent(\ref{17*.5'.b}) If $\theta_{1,2}\in M_{1,2}$, then apply Definition~\ref{17.8'} twice to see that $$(Y\circ w_1)(\theta_{1,2})
=\big(z_1\circ (\proj_{M_{1,1}}\circ \a_1)^{-1}\circ 
\proj_{M_{1,1}}\circ m_2\big)\big(\theta_{1,2}\cdot \a_1(\sigma)\big).$$
The product rule yields that $$m_2\big(\theta_{1,2}\cdot \a_1(\sigma)\big)=
m_1(\theta_{1,2})\cdot \a_1(\sigma)- m_1(\a_1(\sigma))\cdot \theta_{1,2}.$$ The projection map $\proj_{M_{1,1}}$ acts like the identity map on $\a_1(\sigma)$ and like the zero map on 
$\theta_{1,2}$. Thus
\begin{align*}(Y\circ w_1)(\theta_{1,2})&{}=m_1(\theta_{1,2})\cdot z_1\circ (\proj_{M_{1,1}}\circ \a_1)^{-1}
(\a_1(\sigma))\\&{}=m_1(\theta_{1,2})\cdot \sigma\w \sigma=0.\end{align*}

\ms \noindent
(\ref{17*.5'.c}) Let  $\theta_{1,2}\in M_{1,2}$. According to Definition~\ref{17.8'},
\begin{align*}(Y\circ X)(\theta_{1,2})={}&\big(z_1\circ (\proj_{M_{1,1}}\circ \a_1)^{-1}\circ 
\proj_{M_{1,1}}\circ m_2\circ X\big)(\theta_{1,2}).\end{align*}
Apply Hypothesis~\ref{17.9}.(\ref{17.9.b}) to write
$$m_2\circ X(\theta_{1,2})=\beta_0(1)\cdot \theta_{1,2}
-(\a_1\circ \beta_1)(\theta_{1,2}).$$ Recall that $\proj_{M_{1,1}}$ sends $\theta_{1,2}$ to zero and acts like the identity map on the image of $\a_1$.
It follows that 
\begin{align*}(Y\circ X)(\theta_{1,2}){}={}&
-\big(z_1\circ (\proj_{M_{1,1}}\circ \a_1)^{-1}
\circ\a_1\circ \beta_1\big)(\theta_{1,2})\\{}={}&-(z_1\circ \beta_1)(\theta_{1,2})\\
{}={}&-(\beta_2\circ w_1)(\theta_{1,2}),&&\text{by Obs.~\ref{17*.4}.}\end{align*}

\ms\noindent (\ref{18.K}) Use the definition of $Y$, given in \ref{17.8'}, and the Commutative Diagram \ref{17*.4} to see that
\begin{align*}\a_2\circ Y={}&\a_2\circ z_1\circ (\proj_{M_{1,1}}\circ \a_1)^{-1}\circ 
\proj_{M_{1,1}}\circ m_2\\
={}&w_1\circ \a_1\circ (\proj_{M_{1,1}}\circ \a_1)^{-1}\circ 
\proj_{M_{1,1}}\circ m_2
.\end{align*}
The map $\a_1\circ (\proj_{M_{1,1}}\circ \a_1)^{-1}$ is the identity on $M_{1,1}$. Thus, 
\begin{equation}\a_2\circ Y=w_1\circ  
\proj_{M_{1,1}}\circ m_2.\label{Nov-18}\end{equation}

\ms\noindent 
(\ref{18.J})
The definition of $W$ is given in \ref{17.8'}. Observe that
\begin{align*}&W\circ \beta_2+m_3\circ \proj_{M_{3,2}}\circ w_2\\
{}={}&(m_3\circ(\beta_3|_{M_{3,1}})^{-1}\circ z_2)\circ \beta_2+m_3\circ \proj_{M_{3,2}}\circ w_2\\
{}={}&m_3\circ
(\beta_3|_{M_{3,1}})^{-1}\circ \beta_3
\circ w_2+m_3\circ \proj_{M_{3,2}}\circ w_2,&&\text{by Observation~\ref{17*.4},}\\ {}={}&m_3\circ(\proj_{M_{3,1}}+\proj_{M_{3,2}})\circ w_2=m_3\circ w_2,
&&\text{by Observation~\ref{17.3}.(\ref{18.I}).}
\end{align*}

\ms\noindent (\ref{18.O}) Recall the definition of $W$ from 
\ref{17.8'}. 
We calculate the value of $$\proj_{M_{3,2}}\circ w_2\circ W=\proj_{M_{3,2}}\circ w_2\circ m_3\circ(\beta_3|_{M_{3,1}})^{-1}\circ z_2.$$
Apply \ref{17*.3}.(\ref{17*.3.b}) to write
$$w_2\circ m_3= m_4\circ w_3+k_1(\sigma)\cdot \id_{M_3}.$$ Recall from Observation~\ref{17.3}.(\ref{18.N}) that $w_3\circ (\beta_3|_{M_{3,1}})^{-1}\circ z_2=0$. Observe also, that
the image of $\id_{M_3}\circ (\beta_3|_{M_{3,1}})^{-1}$ is contained in $M_{3,1}$; hence 
$$\proj_{M_{3,2}}\circ \id_{M_3}\circ (\beta_3|_{M_{3,1}})^{-1}$$ is the zero map.

\ms\noindent (\ref{18.P}) Observe that 
\begin{align*}
&\proj_{M_{3,2}}\circ(X^\dagger\circ W)\\={}&
\proj_{M_{3,2}}\circ(X^\dagger\circ m_3\circ(\beta_3|_{M_{3,1}})^{-1}\circ z_2),&&\text{by \ref{17.8'},}\\
={}&
\proj_{M_{3,2}}\circ\Big(
\big(-\a_3\circ \beta_3+\beta_0(1)\cdot \id_{M_3}\big)
\circ(\beta_3|_{M_{3,1}})^{-1}\circ z_2\Big),&&\text{by \ref{17*.2}.(\ref{17*.2.d}),}\\
={}&
-\proj_{M_{3,2}}\circ
\a_3\circ \beta_3
\circ(\beta_3|_{M_{3,1}})^{-1}\circ z_2,\end{align*}
because $\proj_{M_{3,2}}(M_{3,1})=0$.
Use the fact that $\beta_3
\circ(\beta_3|_{M_{3,1}})^{-1}=\id_{K_3}$, together with Commutative Diagram \ref{17*.4},
to see that $$\proj_{M_{3,2}}\circ
\a_3\circ \beta_3
\circ(\beta_3|_{M_{3,1}})^{-1}\circ z_2=\proj_{M_{3,2}}\circ
\a_3\circ z_2=\proj_{M_{3,2}}\circ
w_2\circ \a_2.$$

\ms \noindent (\ref{18.M})  
Observe that
\begin{align*}
\beta_2\circ W{}={}&\beta_2\circ m_3\circ(\beta_3|_{M_{3,1}})^{-1}\circ z_2, &&\text{by  \ref{17.8'},}\\
{}={}&k_3\circ \beta_3\circ(\beta_3|_{M_{3,1}})^{-1}\circ z_2, &&\text{by (\ref{17.6}),}
\\
{}={}&k_3\circ  z_2
 \end{align*} 
and 
\begin{align}
     Y\circ \a_2
{}={}&z_1\circ (\proj_{M_{1,1}}\circ \a_1)^{-1}\circ \proj_{M_{1,1}}\circ m_2\circ \a_2,
&&\text{by \ref{17.8'},}\notag
\\\notag
{}={}&z_1\circ (\proj_{M_{1,1}}\circ \a_1)^{-1}\circ \proj_{M_{1,1}}\circ \a_1\circ k_2,
&&\text{by (\ref{17.6}),}
\\
{}={}&z_1\circ  k_2.\label{nov-18-2}
 \end{align} 
Apply Observation~\ref{17*.3}.(\ref{17*.3.a}) to conclude that
$$\beta_2\circ W-Y\circ \a_2=k_3\circ  z_2-z_1\circ  k_2=k_1(\sigma)\cdot \id_{K_2}.$$
\end{proof}

\begin{lemma}
\label{17.47} In the language of Definition~{\rm\ref{17.42}} and Theorem~{\rm\ref{17.9}}, the following identities hold{\rm:}
\begin{enumerate}[\rm(a)]
\item \label{17.47.a}$(r\beta_2-Y)\circ \a_2+k_3\circ z_2=f\cdot\id_{K_2}$,
\item \label{17.47.b}
$\proj_{M_{1,2}}\circ m_2\circ (rX-w_1)|_{M_{1,2}}=f\cdot\id_{M_{1,2}}$,
\item\label{17.47.c}$ -w_3\circ m_4 +r\a_4\circ \beta_4=f\cdot\id_{M_4}$, and 
\item \label{17.47.d}
$(rX-w_1)|_{M_{1,2}}\circ \proj_{M_{1,2}}\circ m_2
+\a_2 \circ(r\beta_2-Y)+m_3\circ (rX^\dagger+w_2)=f\cdot\id_{M_2}$.
\end{enumerate}
\end{lemma}

\begin{proof}(\ref{17.47.a}) 
Recall, from Observation~\ref{17.3}.(\ref{17.3.a}), that
$\beta_2\circ \a_2=\beta_0(1)\cdot \id_{K_2}$. We calculated in (\ref{nov-18-2}) that 
$Y\circ \a_2=z_1\circ k_2$. Recall from Observation \ref{17*.3}.(\ref{17*.3.a}) that
$$-z_1\circ  k_2+k_3\circ z_2=k_1(\sigma)\cdot \id_{K_2}.$$ Use (\ref{oct28}).

\ms \noindent (\ref{17.47.b}) If $\theta_{1,2}\in M_{1,2}$, then 
$$(m_2\circ X)(\theta_{1,2})=\beta_0(1)\cdot \theta_{1,2} -(\a_1\circ \beta_1)(\theta_{1,2})$$ by Hypothesis~\ref{17.9}.(\ref{17.9.b}),
and 
$$(m_2\circ w_1)(\theta_{1,2})=m_1(\theta_{1,2})\cdot \a_1(\sigma)-k_1(\sigma)\cdot \theta_{1,2}$$ by Observation~\ref{17*.3}.(\ref{17*.3.b}).
 The projection map $\proj_{M_{1,2}}$ acts like the identity on $\theta_{1,2}$ but annihilates the image of $\a_1$. Thus, 
$$\proj_{M_{1,2}}\circ m_2\circ (rX-w_1)|_{M_{1,2}}=r\beta_0(1)\cdot \id_{M_{1,2}}+k_1(\sigma)\cdot \id_{M_{1,2}}=f\cdot \id_{M_{1,2}}.$$

\ms \noindent 
(\ref{17.47.c}) Apply 
Observation~\ref{17*.3}.(\ref{17*.3.b}) to see that $-w_3\circ m_4=k_1(\sigma)\cdot \id_{M_4}$. Let $\theta_4$ be an element of $M_4$. Notice that
$$(\a_4\circ \beta_4)(\theta_4)=1\cdot (\a_4\circ \beta_4)(\theta_4)
=(\a_0\circ\beta_0)(1)\cdot \theta_4=\beta_0(1)\cdot \theta_4,$$ by Observation~\ref{17.3}.(\ref{17.3.b}). Hence,
$$-w_3\circ m_4 +r\a_4\circ \beta_4=(k_1(\sigma)+r\beta_0(1))\cdot \id_{M_4}=f\cdot \id_{M_4}.$$

\ms \noindent 
(\ref{17.47.d}) Hypothesis~\ref{17.9}.(\ref{17.9.a}) states that $X|_{M_{1,1}}$ is identically zero; consequently, $$X|_{M_{1,2}}\circ \proj_{M_{1,2}}=X.$$
Thus, \begin{align*} 
&(rX-w_1)|_{M_{1,2}}\circ \proj_{M_{1,2}}\circ m_2
+\a_2 \circ(r\beta_2-Y)+m_3\circ (rX^\dagger+w_2)\\
{}={}&r(X\circ m_2+ \a_2 \circ\beta_2 +m_3\circ X^\dagger)  -w_1\circ \proj_{M_{1,2}}\circ m_2
-\a_2 \circ Y+m_3\circ w_2\\
{}={}&r\beta_0(1)\cdot \id_{M_2} -w_1\circ \proj_{M_{1,2}}\circ m_2
-\a_2 \circ Y+m_3\circ w_2\end{align*}
The most recent equality is due to Hypothesis~\ref{17.9}.(\ref{17.9.c}). 
Recall from (\ref{Nov-18}) that
$$\a_2 \circ Y=w_1\circ  
\proj_{M_{1,1}}\circ m_2.$$
It follows that
\begin{align*} 
&(rX-w_1)|_{M_{1,2}}\circ \proj_{M_{1,2}}\circ m_2
+\a_2 \circ(r\beta_2-Y)+m_3\circ (rX^\dagger+w_2)\\
{}={}&r\beta_0(1)\cdot \id_{M_2} -w_1\circ \proj_{M_{1,2}}\circ m_2
-\a_2 \circ Y+m_3\circ w_2\\
{}={}&r\beta_0(1)\cdot \id_{M_2} -w_1\circ \proj_{M_{1,2}}\circ m_2
-w_1\circ  
\proj_{M_{1,1}}\circ m_2+m_3\circ w_2
\\
{}={}&r\beta_0(1)\cdot \id_{M_2} -w_1\circ  m_2+m_3\circ w_2=
r\beta_0(1)\cdot \id_{M_2}+k_1(\sigma)\cdot\id_{M_2}\\
{}={}&f \cdot\id_{M_2}. \end{align*}
The penultimate equality is established in Observation \ref{17*.3}.(\ref{17*.3.b}).
\end{proof}

\section{The proof of Theorem~\ref{17.9}.}

\begin{chunk}\label{proof17.9} {\it The proof of Theorem~{\rm\ref{17.9}.}}
The proof follows quickly from the  calculations of Sections~\ref{PC} and \ref{PC2}.

\ms\noindent(M.F.\ref{19.9.1}) Observe first that
$(g_{\text{\rm even}}g_{\text{\rm odd}})_{(1,1)}$ is equal to \begin{align*}&(rX-w_1)|_{M_{1,2}}\circ \proj_{M_{1,2}}\circ m_2
+\a_2 \circ(r\beta_2-Y)+m_3\circ (rX^\dagger+w_2)\\{}={}&f\cdot\id_{M_2},\end{align*}
by Lemma \ref{17.47}.(\ref{17.47.d}).  Observe further  that
\begingroup\allowdisplaybreaks 
\begin{align*}
(g_{\text{\rm even}}g_{\text{\rm odd}})_{(1,2)}&{}=-\a_2\circ k_3+m_3\circ \a_3=0,\quad\text{by (\ref{17.6}),}\\
(g_{\text{\rm even}}g_{\text{\rm odd}})_{(1,3)}&{}=m_3\circ m_4=0,\quad\text{by (\ref{17.45}),}\\
(g_{\text{\rm even}}g_{\text{\rm odd}})_{(2,1)}&{}=
z_2\circ Y+
r(-z_2\circ \beta_2+\beta_3\circ w_2)
+r^2\beta_3\circ X^\dagger=0,\\
\intertext{by Observation \ref{17*.4} and Observation \ref{17*.2}.(\ref{17*.2.a}). The homomorphism $Y$ is defined in (\ref{17.8'}). The composition $z_2\circ Y$ is zero because
 $z_2\circ z_1=0$. Observe also that} 
(g_{\text{\rm even}}g_{\text{\rm odd}})_{(2,2)}&{}=
r\beta_3\circ \a_3+z_2\circ k_3-k_4\circ z_3=f\cdot \id_{K_3},\\\intertext{by Observation~\ref{17.3}.(\ref{17.3.a}), Observation \ref{17*.3}.(\ref{17*.3.a}), and (\ref{oct28}),}
(g_{\text{\rm even}}g_{\text{\rm odd}})_{(2,3)}&{}=r(\beta_3\circ m_4-k_4\circ \beta_4)=0,
\quad\text{by (\ref{17.6}),}\\
(g_{\text{\rm even}}g_{\text{\rm odd}})_{(3,1)}&{}=-rw_3\circ X^\dagger-w_3\circ w_4=0,
\quad\text{by \ref{17*.2}.(\ref{17*.2.c}) and \ref{17*.4},}\\
(g_{\text{\rm even}}g_{\text{\rm odd}})_{(3,2)}&{}=-w_3\circ \a_3+\a_4\circ z_3=0, \quad\text{by \ref{17*.4},}\\
(g_{\text{\rm even}}g_{\text{\rm odd}})_{(3,3)}&{}= -w_3\circ m_4 +r\a_4\circ \beta_4=f\cdot \id_{M_4},\quad\text{by \ref{17.47}.(\ref{17.47.c}),}\\
(g_{\text{\rm odd}}g_{\text{\rm even}})_{(1,1)}
&{}=\proj_{M_{1,2}}\circ m_2\circ (rX-w_1)|_{M_{1,2}}=f\cdot\id_{M_{1,2}},\quad\text{by \ref{17.47}.(\ref{17.47.b}),}\\
(g_{\text{\rm odd}}g_{\text{\rm even}})_{(1,2)}&{}=\proj_{M_{1,2}}\circ m_2\circ \a_2=\proj_{M_{1,2}}\circ \a_1 \circ k_2=0,\\\intertext{by (\ref{17.6}) and (\ref{17.46}),}
(g_{\text{\rm odd}}g_{\text{\rm even}})_{(1,3)}
&{}=\proj_{M_{1,2}}\circ m_2\circ m_3=0,\quad\text{by (\ref{17.45}),}\\
(g_{\text{\rm odd}}g_{\text{\rm even}})_{(1,4)}&{}=0\\
(g_{\text{\rm odd}}g_{\text{\rm even}})_{(2,1)}&{}=
(r^2\beta_2\circ X-r(\beta_2\circ w_1+Y\circ X)+Y\circ w_1)|_{M_{1,2}}=0,\\\intertext{by Observation \ref{17*.2}.(\ref{17*.2.b}), and items (\ref{17*.5'.c}), and   (\ref{17*.5'.b}) of Lemma~\ref{17*.5'},}
(g_{\text{\rm odd}}g_{\text{\rm even}})_{(2,2)}&{} = (r\beta_2-Y)\circ \a_2+k_3\circ z_2=f\cdot \id_{K_2}, \quad\text{by \ref{17.47}.(\ref{17.47.a}),}\\
(g_{\text{\rm odd}}g_{\text{\rm even}})_{(2,3)}&{} =r(\beta_2\circ m_3-k_3\circ \beta_3) -Y\circ m_3=0,\\\intertext{by (\ref{17.6}), \ref{17.8'}, and (\ref{17.45}),}
(g_{\text{\rm odd}}g_{\text{\rm even}})_{(2,4)}&{}=k_3\circ k_4=0,\quad\text{by (\ref{17.48})}\\
(g_{\text{\rm odd}}g_{\text{\rm even}})_{(3,1)}&{}=r^2 X^\dagger\circ X+r(w_2\circ X-X^\dagger\circ w_1)-(w_2\circ w_1)=0,\\
\intertext{by Hypothesis \ref{17.9}.(\ref{17.9.d}), Lemma~\ref{17*.5'}.(\ref{17*.5'.a}), and Observation \ref{17*.4},}
(g_{\text{\rm odd}}g_{\text{\rm even}})_{(3,2)}&{}=(rX^\dagger+w_2)\circ \a_2-\a_3\circ z_2=0,\\\intertext{by Hypothesis \ref{17.9}.(\ref{17.9.e}) and Observation \ref{17*.4},}
(g_{\text{\rm odd}}g_{\text{\rm even}})_{(3,3)}&{}=r(X^\dagger\circ m_3+\a_3\circ \beta_3)+(w_2\circ m_3-m_4\circ w_3)=f\cdot \id_{M_3},\\\intertext{
by Observation \ref{17*.2}.(\ref{17*.2.d}) and Observation \ref{17*.3}.(\ref{17*.3.b}),}
(g_{\text{\rm odd}}g_{\text{\rm even}})_{(3,4)}&{}=-\a_3\circ k_4+m_4\circ \a_4=0,\quad\text{by (\ref{17.6}),}\\
(g_{\text{\rm odd}}g_{\text{\rm even}})_{(4,1)}&{}=0,\\
(g_{\text{\rm odd}}g_{\text{\rm even}})_{(4,2)}&{}=-z_3\circ z_2 =0,\quad\text{by \ref{17*.4}}\\
(g_{\text{\rm odd}}g_{\text{\rm even}})_{(4,3)}&=r(z_3\circ \beta_3-\beta_4\circ w_3)=0,\quad\text{by \ref{17*.4}, and}\\
(g_{\text{\rm odd}}g_{\text{\rm even}})_{(4,4)}&=-z_3\circ k_4+r\beta_4\circ \a_4=f\cdot\id_{K_4},\quad\intertext{by Observation \ref{17*.3}.(\ref{17*.3.a}) and Observation \ref{17.3}.(\ref{17.3.a}).} 
\end{align*}\endgroup

\ms\noindent(M.F.\ref{19.9.2}) The product $\acute{g}_{\text{\rm even}}\acute{g}_{\text{\rm odd}}$ is equal to $rA+B+r^{-1}C$, where
\begin{align*}
A{}={}&{X\circ \proj_{M_{1,2}}}\circ m_2+\a_2\circ \beta_2+ m_3\circ\proj_{M_{3,2}}\circ X^\dagger,\\
B{}={}&-w_1\circ \proj_{M_{1,2}}\circ m_2-\a_2\circ Y
+W\circ \beta_2+m_3\circ\proj_{M_{3,2}}\circ w_2, \text{ and}\\
C{}={}&-W\circ Y.\end{align*}
Recall from Hypothesis~\ref{17.9}.(\ref{17.9.a}) and Observation~\ref{17*.2}.(\ref{18.H}) that $$XM_{1,1}=0\quad\text{and}\quad \im X^\dagger\subseteq M_{3,2}.$$ It follows that 
$X\circ \proj_{M_{1,2}}=X$ and $\proj_{M_{3,2}}\circ X^\dagger= X^\dagger$. Thus,
$$A=X\circ m_2+\a_2\circ \beta_2+ m_3\circ X^\dagger=\beta_0(1)\cdot \id_{M_2}.$$ The final equality is due to Hypothesis~\ref{17.9}.(\ref{17.9.c}).

The first two terms of $B$ add to $-w_1\circ m_2$ and the last two terms add to $m_3\circ w_2$ by items (\ref{18.K}) and 
(\ref{18.J}), respectively, of Lemma~\ref{17*.5'}. Apply Observation~\ref{17*.3}.(\ref{17*.3.b}) to conclude that $B=k_1(\sigma)\cdot \id_{M_2}$.

The composition $W\circ Y$ factors through $z_2\circ z_1=0$ (see Definition~\ref{17.8'}); hence, $C=0$ and $$\acute{g}_{\text{\rm even}}\acute{g}_{\text{\rm odd}}=(r\beta_0(1)+k_1(\sigma))\cdot \id_{M_2}=f\cdot \id_{M_2}.$$

We compute the composition $\acute{g}_{\text{\rm odd}}\acute{g}_{\text{\rm even}}$. Observe that 
$$(\acute{g}_{\text{\rm odd}}\acute{g}_{\text{\rm even}})_{1,1}= \proj_{M_{1,2}}\circ m_2\circ (rX-w_1)|_{M_{1,2}}.$$
Apply Hypothesis~\ref{17.9}.(\ref{17.9.b}) and Observation~\ref{17*.3}.(\ref{17*.3.b}) to write
\begin{align*}m_2\circ X&{}= \beta_0(1)\cdot \id_{M_1}-\a_1\circ \beta_1\quad\text{and}\\
-m_2\circ w_1&{}= -w_0\circ m_1 +k_1(\sigma)\cdot \id_{M_1}.\end{align*}
Recall, from (\ref{17.46}), that $\proj_{M_{1,2}}\circ\a_1=0$. Notice that the image of $w_0\circ m_1$ is contained in $M_{1,1}$; hence, $\proj_{M_{1,2}}\circ w_0\circ m_1=0$. Thus,
$$(\acute{g}_{\text{\rm odd}}\acute{g}_{\text{\rm even}})_{1,1}=(r\beta_0(1)+k_1(\sigma))\cdot \id_{M_{1,2}}=f\cdot \id_{M_{1,2}}.$$

The map $(\acute{g}_{\text{\rm odd}}\acute{g}_{\text{\rm even}})_{1,2}$ is equal to
$$
\proj_{M_{1,2}}\circ m_2\circ (\a_2+r^{-1}W).
$$
Apply (\ref{17.6}) and (\ref{17.46}) to see that
$$\proj_{M_{1,2}}\circ m_2\circ \a_2=\proj_{M_{1,2}}\circ \a_1\circ k_2=0$$
by (\ref{17.6}) and (\ref{17.46}). 
The equality $m_2\circ W=0$ follows immediately from the definition of $W$ in \ref{17.8'}. It follows that $(\acute{g}_{\text{\rm odd}}\acute{g}_{\text{\rm even}})_{1,2}=0$. 

The map $(\acute{g}_{\text{\rm odd}}\acute{g}_{\text{\rm even}})_{1,3}$ is 
$$\proj_{M_{1,2}}\circ m_2\circ m_3|_{M_{3,2}}=0.$$ The map
$(\acute{g}_{\text{\rm odd}}\acute{g}_{\text{\rm even}})_{2,1}$ is
$$r^2\beta_2\circ X|_{M_{1,2}}-r(\beta_2\circ w_1+Y\circ X)|_{M_{1,2}}+Y\circ w_1|_{M_{1,2}}=0$$
by 
\ref{17*.2}.(\ref{17*.2.b}), and items (\ref{17*.5'.c}), and (\ref{17*.5'.b}) of Lemma~\ref{17*.5'}. 
Observe that 
\begin{align*}(\acute{g}_{\text{\rm odd}}\acute{g}_{\text{\rm even}})_{2,2} 
{}={}&r\beta_2\circ\a_2+(\beta_2\circ W-Y\circ \a_2)+r^{-1}Y\circ W\\
{}={}&r\beta_0(1)\cdot \id_{K_2}+k_1(\sigma)\cdot \id_{K_2}+r^{-1}\cdot 0=f\cdot \id_{K_2}\end{align*}
by Observation~\ref{17.3}.(\ref{17.3.a}), Lemma~\ref{17*.5'}.(\ref{18.M}), and the fact that $Y\circ W$ factors through $$m_2\circ m_3=0;$$ see Definition~\ref{17.8'}. The map 
$(\acute{g}_{\text{\rm odd}}\acute{g}_{\text{\rm even}})_{2,3}$ equals
$$r
\beta_2\circ m_3|_{M_{3,2}}-
Y\circ m_3|_{M_{3,2}}=0.$$
Indeed, $\beta_2\circ m_3=k_3\circ \beta_3$ by Commutative Diagram \ref{17.6}, $\beta_3|_{M_{3,2}}=0$ by Observation~\ref{17.3}.(\ref{18.D}), and $Y\circ m_3$ factors through $m_2\circ m_3=0$ by the definition of $Y$, see \ref{17.8'}.
Observe that
$$(\acute{g}_{\text{\rm odd}}\acute{g}_{\text{\rm even}})_{3,1}=
\proj_{M_{3,2}}\circ \Big(r^2 X^\dagger\circ X+r(w_2\circ X-X^\dagger\circ w_1)-w_2\circ w_1 \Big)|_{M_{1,2}}
$$
and this is zero by Hypothesis~\ref{17.9}.(\ref{17.9.d}), Lemma~\ref{17*.5'}.(\ref{17*.5'.a}), and the fact that the rows  Commutative Diagram~\ref{17*.4} are complexes.
Apply Hypothesis~\ref{17.9}.(\ref{17.9.e}) and items  (\ref{18.O}) and (\ref{18.P}) of  Lemma~\ref{17*.5'} to see that
$$(\acute{g}_{\text{\rm odd}}\acute{g}_{\text{\rm even}})_{3,2}= 
\proj_{M_{3,2}}\circ(r X^\dagger\circ \a_2 + (X^\dagger\circ W+ w_2\circ \a_2 )+r^{-1}w_2\circ W)
$$ is equal to zero.
The map $(\acute{g}_{\text{\rm odd}}\acute{g}_{\text{\rm even}})_{3,3}$ is equal to
\begin{align*}
&\proj_{M_{3,2}}\circ(rX^\dagger+w_2)\circ m_3|_{M_{3,2}}\\
{}={}&\proj_{M_{3,2}}\circ\big(r(\beta_0(1)\cdot \id_{M_3}-\a_3\circ \beta_3)
+m_4\circ w_3+k_1(\sigma)\cdot \id_{M_3}\big)|_{M_{3,2}}\end{align*}
by \ref{17*.2}.(\ref{17*.2.d}) and \ref{17*.3}.(\ref{17*.3.b}).
Recall from Observation~\ref{17.3}.(\ref{18.D}) that $\beta_3(M_{3,2})=0$. Recall, also, from \ref{17.8'} and \ref{17.7}, that $w_3(M_{3,2})\subseteq M_{3,2}\cdot M_{1,1}=0$. 
It follows that $(\acute{g}_{\text{\rm odd}}\acute{g}_{\text{\rm even}})_{3,3}$
 is $$\proj_{M_{3,2}}\circ\big(r\beta_0(1)\cdot \id_{M_3}
+k_1(\sigma)\cdot \id_{M_3}\big)|_{M_{3,2}}= f\cdot \id_{M_{3,2}}.$$
\vskip-24pt \hfill \qed
\end{chunk}

\section{The matrix factorization of Theorem \ref{17.9}  induces the infinite tail of the resolution of $P/(f,\mfK)$ by free $P/(f)$ modules.}
Let $\Pbar$ represent $P/(f)$ and $\ov{\phantom{x}}$ represent the functor $-\t_P\Pbar$. 

\begin{theorem}\label{thm17} Adopt the language of Theorem~{\rm\ref{17.9}}. Then the following statements hold.
\begin{enumerate}[\rm 1.]
\item\label{one} The maps and modules
\begin{equation}\label{17goal} N:\quad \cdots \xrightarrow{n_3}  N_2 \xrightarrow{n_2} N_1\xrightarrow{n_1}N_0\end{equation}
form a resolution of $\Pbar/\mfK\Pbar$ by free $\Pbar$-modules, where
 the modules of $N$ are
$$N_i=\begin{cases}
\ov{K_0},&\text{if $i=0$},\\
\ov{K_1},&\text{if $i=1$},\\
\ov{M_{1,2}}\p \ov{K_2},&\text{if $i=2$},\\
\ov{M_2}\p \ov{K_3},&\text{if $i=3$},\\
\ov{G_\text{\rm even}}, &\text{if $4\le i$ and $i$ is even, and}\\ 
\ov{G_\text{\rm odd}}, &\text{if $5\le i$ and  $i$ is odd,}
\end{cases}$$
and 
 the differentials $n_i$ are given by 
\begingroup\allowdisplaybreaks
\begin{align*} 
n_1&{}=-\ov{k_1};\\[5pt]
n_2&{}=\bmatrix (r\ov{\beta_1}+\ov{z_0\circ m_1})|_{\ov{M_{1,2}}}&-\ov{k_2}
\endbmatrix;\\[5pt]
n_3&{}=\bmatrix 
\proj_{M_{1,2}}\circ \ov{m_2}&0 \\
r\ov{\beta_2}-\ov{Y}&-\ov{k_3}
\endbmatrix;\\[5pt]
n_4&{}=\bmatrix 
(r\ov{X}-\ov{w_1})|_{\ov{M_{1,2}}}&\ov{\a_2}&\ov{m_3}&0\\
0&-\ov{z_2}&r\ov{\beta_3}&-\ov{k_4}\\
\endbmatrix;\\
n_i&{}=\ov{g_{\text{\rm odd}}},&&\text{if $5\le i$ and $i$ is odd; and}
\\[5pt]
n_i&{}=\ov{g_{\text{\rm even}}},&&\text{if $6\le i$ and $i$ is even.}
\end{align*}
\endgroup 
\item\label{two} If $r$ is a unit, then the maps and modules
\begin{equation}\label{17goal'} \acute{N}:\quad \cdots \xrightarrow{\acute{n}_3}  \acute{N}_2 \xrightarrow{\acute{n}_2} \acute{N}_1\xrightarrow{\acute{n}_1}\acute{N}_0\end{equation}
form a resolution of $\Pbar/\mfK\Pbar$ by free $\Pbar$-modules, where
 the modules of $\acute{N}$ are
$$\acute{N}_i=\begin{cases}
\ov{K_0},&\text{if $i=0$},\\
\ov{K_1},&\text{if $i=1$},\\
\ov{M_{1,2}}\p \ov{K_2},&\text{if $i=2$},\\
\ov{\acute{G}_\text{\rm odd}}, &\text{if $3\le i$ and  $i$ is odd,}\\
\ov{\acute{G}_\text{\rm even}}, &\text{if $4\le i$ and $i$ is even, and}\\ 
\end{cases}$$
and 
 the differentials $\acute{n}_i$ are given by 
\begingroup\allowdisplaybreaks
\begin{align*} 
\acute{n}_1&{}=-\ov{k_1};\\[5pt]
\acute{n}_2&{}=\bmatrix (r\ov{\beta_1}+\ov{z_0\circ m_1})|_{\ov{M_{1,2}}}&-\ov{k_2}
\endbmatrix;\\[5pt]
\acute{n}_3&{}=\bmatrix 
\proj_{M_{1,2}}\circ \ov{m_2} \\
r\ov{\beta_2}-\ov{Y}
\endbmatrix;\\[5pt]
\acute{n}_i&{}=\ov{\acute{g}_{\text{\rm even}}},&&\text{if $4\le i$ and $i$ is even; and}
\\
\acute{n}_i&{}=\ov{\acute{g}_{\text{\rm odd}}},&&\text{if $5\le i$ and $i$ is odd.}
\\[5pt]
\end{align*}
\endgroup 
\end{enumerate}
\end{theorem}

\begin{proof} The idea for this proof is inspired by the proof of \cite[Lem.~2.3]{KRV}.
Recall the map of complexes $\beta:M\to K$ of Observation~\ref{17.41}. Consider the perturbation $\beta':M\to K$ of $\beta$, where
\begin{equation}\label{p355}\beta_i'=\begin{cases} r\beta_i,&\text{for $2\le i\le 4$,}\\
r\beta_1+z_0\circ m_1,&\text{for $i=1$,}\\
r\beta_0+m_1\circ w_0,&\text{for $i=0$},\end{cases}\end{equation}for $r$  defined in (\ref{17.3.2}) and 
$z_1$ and $w_0$ defined in (\ref{17.8'}). In particular,
\begin{equation}\label{f-oct-31}\beta_0'(1)
=r\beta_0(1)+k_1(\sigma)=f.\end{equation} 
It is easy to see that $\beta':M\to K$ is also a map of complexes. Indeed, the only interesting calculation occurs in the right most square; and this square commutes because
$$m_1\circ w_0\circ m_1=k_1\circ z_0\circ m_1,$$
since 
\begin{align*}
m_1\circ (w_0\circ m_1)&{}=m_1\circ (k_1(\sigma)\cdot \id_{M_1}+m_2\circ w_1)=k_1(\sigma)\cdot m_1&&\text{by \ref{17*.3}.(\ref{17*.3.b}), and}\\
(k_1\circ z_0)\circ m_1&{}=(k_1(\sigma)\cdot \id_{K_0})\circ m_1=k_1(\sigma)\cdot m_1&&\text{by \ref{17*.3}.(\ref{17*.3.a}).}\end{align*}

Consider the short exact sequence 
$$0\to P/(\mfK:f)\to P/\mfK\to P/(\mfK,f)\to 0.$$
The complexes $M$ and $K$ are resolutions of $P/(\mfK:f)$ and $P/\mfK$, respectively, by free $P$-modules. It follows that the 
  mapping cone $L$ of 
$$\xymatrix{
0\ar[r]& M_4\ar[r]^{m_4}\ar[d]^{\beta'_4}& M_3\ar[r]^{m_3}\ar[d]^{\beta'_3}& M_2\ar[r]
^{m_2}\ar[d]^{\beta'_2}& M_1   \ar[r]^{m_1}\ar[d]^{\beta'_1}& M_0\ar[d]^{\beta'_0}
\\
0\ar[r]& K_4\ar[r]^{k_4}& K_3\ar[r]^{k_3}& K_2\ar[r]^{k_2}& K_1   \ar[r]^{k_1}& K_0
}$$
is a resolution of $P/(\mfK,f)$ by free $P$-modules. 
This resolution has the form  $$L:\quad 0\to L_5\xrightarrow{\ell_5} L_4\xrightarrow{\ell_4}
L_3\xrightarrow{\ell_3}
L_2\xrightarrow{\ell_2}L_1\xrightarrow{\ell_1}L_0,$$
 where
 $$L_5=M_4,\quad L_4=\begin{matrix}M_3\\\p\\
K_4,\end{matrix}\quad
\quad L_3=\begin{matrix} M_2\\\p\\K_3,\end{matrix}\quad
L_2=\begin{matrix} M_1\\\p\\K_2,\end{matrix}\quad
L_1= \begin{matrix} M_0\\\p\\K_1,\end{matrix}\quad
L_0=K_0,$$
$$\ell_5=\bmatrix m_4\\\beta'_4\endbmatrix, \quad
\ell_4=\bmatrix m_3&0\\\beta'_3&-k_4\endbmatrix, \quad
\ell_3=\bmatrix m_2&0\\\beta'_2&-k_3 \endbmatrix,\quad
\ell_2=\bmatrix m_1&0\\\beta'_1&-k_2\endbmatrix,$$and
 $\ell_1=\bmatrix \beta'_0&-k_1 \endbmatrix$. 
The element $f$ of $P$ is regular by hypothesis; hence
$\ov{L}$ is a complex with homology:
$$\HH_i(\Lbar)=\Tor_i^P(\Pbar/\mfK \Pbar,\Pbar)=\begin{cases} \Pbar/\mfK \Pbar, &\text{if $i$ is $0$ or $1$, and }\\0,&\text{otherwise.} \end{cases}$$
Furthermore, the cycle 
\begin{equation}\label{17xi}\xi= \bmatrix 1\\0\endbmatrix\end{equation} in $\overline{L_1}$ represents a generator of $\operatorname{H}_1(\overline{L})$. We kill the homology in $\overline{L}$. Define $P$-module homomorphisms $\XXX_i: L_i\to L_{i+1}$ by 
\begin{align*}\rho_4&{}=\bmatrix -w_3&\a_4\endbmatrix,\quad \XXX_3=\bmatrix -rX^\dagger-w_2&-\a_3\\0&-z_3\endbmatrix,\quad 
\XXX_2=\bmatrix rX-w_1&\a_2\\0&-z_2\endbmatrix,\\
\XXX_1&{}=\bmatrix 0&-\a_1\\0&-z_1\endbmatrix,\quad \text{and}\quad  \XXX_0=\bmatrix \a_0\\0\endbmatrix.\end{align*}
  It is shown in Lemma~\ref{17-oct20}.(\ref{17-oct20.a}) that
\begin{equation}
\label{17ll}\xymatrix{
0\ar[r]&\ov{L_5}\ar[r]^{\ov{\ell_5}}\ar[d]&
\ov{L_4}\ar[r]^{\ov{\ell_4}}\ar[d]^{\ov{\rho_4}}&
\ov{L_3}\ar[r]^{\ov{\ell_3}}\ar[d]^{\ov{\XXX_3}}&
\ov{L_2}\ar[r]^{\ov{\ell_2}}\ar[d]^{\ov{\XXX_2}}&
\ov{L_1}\ar[r]^{\ov{\ell_1}}\ar[d]^{\ov{\XXX_1}}&
\ov{L_0}\ar[r]\ar[d]^{\ov{\XXX_0}}&0\ar[d]\\
&0\ar[r]&\ov{L_5}\ar[r]^{\ov{\ell_5}}&
\ov{L_4}\ar[r]^{\ov{\ell_4}}&
\ov{L_3}\ar[r]^{\ov{\ell_3}}&
\ov{L_2}\ar[r]^{\ov{\ell_2}}&
\ov{L_1}\ar[r]^{\ov{\ell_1}}&
\ov{L_0}\ar[r]&0\\}\end{equation}
is a map of complexes. It is clear that $\XXX_0$ induces an isomorphism from $\operatorname{H}_0$ of the top line of (\ref{17ll}) to $\operatorname{H}_1$ of the bottom  line of (\ref{17ll}). Let $\mathbb M$ be the total complex of (\ref{17ll}). We have shown that the homology of $\mathbb M$ is concentrated in positions $0$ and $3$ and the $\xi$ from (\ref{17xi}) of the summand $\overline{L}_1$ in $\mathbb M_3=\overline{L}_1\oplus \overline{L}_3$ represents the $\ov{L_1}$-component of a generator of $\operatorname{H}_3(\mathbb M)$. 
It is shown in 
Lemma~\ref{17-oct20}.(\ref{17-oct20.b}) that $\ov{\rho_1}\circ \ov{\rho_0}=0$; so indeed, the element $\xi$ of $\mathbb M_3$ is a cycle of $\mathbb M$. We kill the homology  of $\mathbb M$.
In theory we need to give a map of complexes from $\ov{L}[-4]$ to all of $\mathbb M$; however, in practice, because of Lemma~\ref{17-oct20}.(\ref{17-oct20.b}), it suffices to give a map of complexes from $\ov{L}[-4]$ to the top line $\ov{L}[-2]$ of (\ref{17ll}).
Iterate this process to see that  $\Pbar/\mfK\Pbar$ is resolved by the total complex $\mathbb T$ of the 
 infinite double complex given in Table~\ref{17picture}.  
\begin{table}
\begin{center}
\xymatrix{
&\vdots\ar[d]^{\ov{\XXX_4}}
&\vdots\ar[d]^{\ov{\XXX_3}}&
\vdots\ar[d]^{\ov{\XXX_2}}&
\vdots\ar[d]^{\ov{\XXX_1}}&
\vdots\ar[d]^{\ov{\XXX_0}}&
\vdots\ar[d]\\
0\ar[r]&\ov{L_5}\ar[r]^{\ov{\ell_5}}\ar[d]
&\ov{L_4}\ar[r]^{\ov{\ell_4}}\ar[d]^{\ov{\XXX_4}}&
\ov{L_3}\ar[r]^{\ov{\ell_3}}\ar[d]^{\ov{\XXX_3}}&
\ov{L_2}\ar[r]^{\ov{\ell_2}}\ar[d]^{\ov{\XXX_2}}&
\ov{L_1}\ar[r]^{\ov{\ell_1}}\ar[d]^{\ov{\XXX_1}}&
\ov{L_0}\ar[r]\ar[d]^{\ov{\XXX_0}}&0\ar[d]\\
&0\ar[r]
&\ov{L_5}\ar[r]^{\ov{\ell_5}}\ar[d]
&\ov{L_4}\ar[r]^{\ov{\ell_4}}\ar[d]^{\ov{\XXX_4}}
&\ov{L_3}\ar[r]^{\ov{\ell_3}}\ar[d]^{\ov{\XXX_3}}&
\ov{L_2}\ar[r]^{\ov{\ell_2}}\ar[d]^{\ov{\XXX_2}}&
\ov{L_1}\ar[r]^{\ov{\ell_1}}\ar[d]^{\ov{\XXX_1}}&
\ov{L_0}\ar[r]\ar[d]^{\ov{\XXX_0}}&0\ar[d]\\
&&0\ar[r]&
\ov{L_5}\ar[r]^{\ov{\ell_5}}&
\ov{L_4}\ar[r]^{\ov{\ell_4}}&
\ov{L_3}\ar[r]^{\ov{\ell_3}}&
\ov{L_2}\ar[r]^{\ov{\ell_2}}&
\ov{L_1}\ar[r]^{\ov{\ell_1}}&
\ov{L_0}\ar[r]&0.}

\medskip

\caption{The total complex of this infinite double complex is called $\mathbb T$.}\label{17picture}
\end{center}
\end{table}
We emphasize that it is shown in Lemma~\ref{17-oct20}.(\ref{17-oct20.b}) that each column of Table~\ref{17picture} is a complex.
Observe that the modules of $\mathbb T$ are
$$\mathbb T_i=\begin{cases} 
\ov{L_0}&\text{if $i=0$}\\
\ov{L_1}&\text{if $i=1$}\\
\ov{L_0}\p\ov{L_2},&\text{if $i=2$,}\\
\ov{L_1}\p \ov{L_3},&\text{if $i=3$,}\\
\ov{L_0}\p\ov{L_2}\p\ov{L_4},&\text{if $4\le i$ and $i$ is even,}\\
\ov{L_1}\p \ov{L_3}\p \ov{L_5},&\text{if $5\le i$ and $i$ is odd,}\\\end{cases}$$ and the differential of $\mathbb T$ is
\begingroup\allowdisplaybreaks
\begin{align*}t_1&{}=\ov{\ell_1},\quad 
t_2=\bmatrix \ov{\XXX_0}&\ov{\ell_2}\endbmatrix,\quad
t_3=\bmatrix 
\ov{\ell_1}&0\\
-\ov{\XXX_1}&\ov{\ell_3}
\endbmatrix,\quad
t_4=\bmatrix 
\ov{\XXX_0}&\ov{\ell_2}&0\\
0&\ov{\XXX_2}&\ov{\ell_4}\\
\endbmatrix,\\
t_i&{}=\bmatrix 
\ov{\ell_1}&0&0\\
-\ov{\XXX_1}&\ov{\ell_3}&0\\
0&-\ov{\XXX_3}&\ov{\ell_5}
\endbmatrix,\quad\text{if $5\le i$ and $i$ is odd, and}\\
t_i&{}=\bmatrix 
\ov{\XXX_0}&\ov{\ell_2}&0\\
0&\ov{\XXX_2}&\ov{\ell_4}\\
0&0&\ov{\rho_4}
\endbmatrix,\quad\text{if $6\le i$ and $i$ is even.}\end{align*}\endgroup
In order to remove the parts of $\mathbb T$ that obviously split off, we record $\mathbb T$ explicitly and we employ the decomposition $$M_1=M_{1,1}\p M_{1,2}.$$ Thus, $\mathbb T_i$ is equal to
\begingroup\allowdisplaybreaks\begin{align*} 
&\ov{K_0},&&\text{if $i=0$,}\\
&\ov{M_0}\p \ov{K_1},&&\text{if $i=1$,}\\
&\ov{K_0}\p\ov{M_{1,1}}\p\ov{M_{1,2}}\p\ov{K_2},&&\text{if $i=2$,}\\
&\ov{M_0}\p \ov{K_1}\p\ov{M_{2}}\p\ov{K_3},&&\text{if $i=3$,}\\
&\ov{K_0}\p\ov{M_{1,1}}\p\ov{M_{1,2}}\p\ov{K_2}\p\ov{M_{3}}
\p \ov{K_4}
,&&\text{if $4\le i$ and $i$ is even,}\\
&\ov{M_0}\p \ov{K_1}\p\ov{M_{2}}\p\ov{K_3}
\p\ov{M_4}
,&&\text{if $5\le i$ and $i$ is odd,}\\\end{align*}\endgroup
and the differentials $t_i$ are given by 
\begingroup\allowdisplaybreaks
\begin{align} 
t_1&{}=\bmatrix \ov{\beta_0'}&-\ov{k_1}\endbmatrix;\notag\\[5pt]
t_2&{}=\bmatrix \ov{\a_0}&\ov{m_1}|_{\ov{M_{1,1}}}&\ov{m_1}|_{\ov{M_{1,2}}}&0\\
0&\ov{\beta_1'}|_{\ov{M_{1,1}}}
&\ov{\beta_1'}|_{\ov{M_{1,2}}}&-\ov{k_2}
\endbmatrix;\notag\\[5pt]
t_3&{}=\bmatrix 
\ov{\beta_0'}&-\ov{k_1}&0&0\\
0&\proj_{\ov{M_{1,1}}}\circ \ov{\a_1}&\proj_{M_{1,1}}\circ \ov{m_2}&0 \\
0&0&\proj_{\ov{M_{1,2}}}\circ \ov{m_2}&0 \\
0&\ov{z_1}&r\ov{\beta_2}&-\ov{k_3}
\endbmatrix;\notag\\[5pt]
t_4&{}=\bmatrix \ov{\a_0}&\ov{m_1}|_{\ov{M_{1,1}}}&\ov{m_1}|_{\ov{M_{1,2}}}&0&0&0\\
0&\ov{\beta_1'}|_{\ov{M_{1,1}}}
&\ov{\beta_1'}|_{\ov{M_{1,2}}}&-\ov{k_2}&0&0\\
0&-\ov{w_1}|_{\ov{M_{1,1}}}&(r\ov{X}-\ov{w_1})|_{\ov{M_{1,2}}}&\ov{\a_2}&\ov{m_3}&0\\
0&0&0&-\ov{z_2}&r\ov{\beta_3}&-\ov{k_4}\\
\endbmatrix;\notag\\
t_i&{}=\bmatrix 
\ov{\beta_0'}&-\ov{k_1}&0&0&0\\
0&\proj_{\ov{M_{1,1}}}\circ \ov{\a_1}&\proj_{\ov{M_{1,1}}}\circ \ov{m_2}&0&0 \\
0&0&\proj_{\ov{M_{1,2}}}\circ \ov{m_2}&0&0 \\
0&\ov{z_1}&r\ov{\beta_2}&-\ov{k_3}&0\\
0&0&
(r\ov{X^\dagger}+\ov{w_2})
& \ov{\a_3}
&\ov{m_4}\\
0&0&0&\ov{z_3}&r\ov{\beta_4}
\endbmatrix,\notag\\\intertext{if $5\le i$ and $i$ is odd; and}
\notag\\[5pt]
t_i&{}=\bmatrix \ov{\a_0}&\ov{m_1}|_{\ov{M_{1,1}}}&\ov{m_1}|_{\ov{M_{1,2}}}&0&0&0\\
0&\ov{\beta_1'}|_{\ov{M_{1,1}}}
&\ov{\beta_1'}|_{\ov{M_{1,2}}}&-\ov{k_2}&0&0\\
0&-\ov{w_1}|_{\ov{M_{1,1}}}&(r\ov{X}-\ov{w_1})|_{\ov{M_{1,2}}}&\ov{\a_2}&\ov{m_3}&0\\
0&0&0&-\ov{z_2}&r\ov{\beta_3}&-\ov{k_4}\\
0&0&0&0&-\ov{w_3}&\ov{\a_4}
\endbmatrix,\notag 
\end{align}
\endgroup if $6\le i$ and $i$ is even. The maps $\ov{\a_0}$ and $\proj_{\ov{M_{1,1}}}\circ \ov{\a_1}$ are isomorphisms. One applies elementary row and column operations to see that the complex $(\mathbb T,t)$ is isomorphic to the complex $(\mathbb T,t')$
where the differentials $t_i'$ are given by 
\begingroup\allowdisplaybreaks
\begin{align*} 
t_1'&{}=\bmatrix 0&-\ov{k_1}\endbmatrix;\\[5pt]
t_2'&{}=\bmatrix \ov{\a_0}&0&0&0\\
0&0
&\ov{\beta_1'}|_{\ov{M_{1,2}}}&-\ov{k_2}
\endbmatrix;\\[5pt]
t_3'&{}=\bmatrix 
0&0&0&0\\
0&\proj_{M_{1,1}}\circ \ov{\a_1}&0&0 \\
0&0&\proj_{M_{1,2}}\circ \ov{m_2}&0 \\
0&0&r\ov{\beta_2}-\ov{Y}&-\ov{k_3}
\endbmatrix;\\[5pt]
t_4'&{}=\bmatrix \ov{\a_0}&0&0&0&0&0\\
0&0
&0&0&0&0\\
0&0&(r\ov{X}-\ov{w_1})|_{\ov{M_{1,2}}}&\ov{\a_2}&\ov{m_3}&0\\
0&0&0&-\ov{z_2}&r\ov{\beta_3}&-\ov{k_4}\\
\endbmatrix;\\
t_i'&{}=\bmatrix 
0&0&0&0&0\\
0&\proj_{\ov{M_{1,1}}}\circ \ov{\a_1}&0&0&0 \\
0&0&\proj_{\ov{M_{1,2}}}\circ \ov{m_2}&0&0 \\
0&0&r\ov{\beta_2}-\ov{Y}&-\ov{k_3}&0\\
0&0&r \ov{X^\dagger}+\ov{w_2}& \ov{\a_3}
&\ov{m_4}\\
0&0&0&\ov{z_3}&r\ov{\beta_4}
\endbmatrix,\\\intertext{if $5\le i$ and $i$ is odd; and}
\\[5pt]
t_i'&{}=\bmatrix \ov{\a_0}&0&0&0&0&0\\
0&0
&0&0&0&0\\
0&0&(r\ov{X}-\ov{w_1})|_{\ov{M_{1,2}}}&\ov{\a_2}&\ov{m_3}&0\\
0&0&0&-\ov{z_2}&r\ov{\beta_3}&-\ov{k_4}\\
0&0&0&0&-\ov{w_3}&\ov{\a_4}
\endbmatrix,
\end{align*}
\endgroup if $6\le i$ and $i$ is even.

It is clear that the complex $N$ of (\ref{17goal}) is a subcomplex of the resolution $(\mathbb T, t')$ and the inclusion map is a quasi-isomorphism. Thus, $N$ is a resolution. The completes the proof of statement \ref{one}.

The proof of statement \ref{two} begins with the resolution $N$ from \ref{one}. The module $\ov{M_3}$ is now written as $\ov{M_{3,1}}\p \ov{M_{3,2}}$. The differentials $n_1$, $n_2$, and $n_3$ are unchanged, and the other  differentials are now written as follows: 
\begingroup\allowdisplaybreaks
\begin{align*} 
n_4&{}=\ov{\bmatrix 
(rX-w_1)|_{M_{1,2}}&\a_2&m_3|_{M_{3,1}}&m_3|_{M_{3,2}}&0\\
0&-z_2&r\beta_3|_{M_{3,1}}&0&-k_4\\
\endbmatrix};\\
n_i&{}=\ov{
\bmatrix 
\proj_{M_{1,2}}\circ m_2&0&0 \\
r\beta_2-Y&-k_3&0\\
\proj_{M_{3,1}}\circ (rX^\dagger+w_2)&\proj_{M_{3,1}}\circ \a_3&\proj_{M_{3,1}}\circ m_4\\
\proj_{M_{3,2}}\circ (rX^\dagger+w_2)&\proj_{M_{3,2}}\circ \a_3&\proj_{M_{3,2}}\circ m_4\\
0&z_3&r\beta_4
\endbmatrix},\intertext{if $5\le i$ and $i$ is odd; and}
\\[5pt]
n_i&{}=\ov{
\bmatrix 
(rX-w_1)|_{M_{1,2}}&\a_2&m_3|_{M_{3,1}}&m_3|_{M_{3,2}}&0\\
0&-z_2&r\beta_3|_{M_{3,1}}&0&-k_4\\
0&0&-w_3|_{M_{3,1}}&-w_3|_{M_{3,2}}&\a_4
\endbmatrix},
\end{align*}
\endgroup if $6\le i$ and $i$ is even. The map $r\beta_3|_{M_{3,2}}$ should appear in row 2, column 4 of the map $n_i$, for even $i$ with $4\le i$. This map is zero according to Observation~\ref{17.3}.(\ref{18.D}).  
 
Recall from (\ref{19.3.7}) that $r\beta_4: M_4\to K_4$ and $r\beta_3|_{M_{3,1}}:M_{3,1}\to K_3$ are isomorphisms. One uses elementary row and column operations, as was done above, to obtain a complex isomorphic to $N$, which is quasi-isomorphic to $\acute{N}$. 
\end{proof}

The two calculations in the next result were used in the proof of Theorem~\ref{thm17}.

\begin{lemma}\label{17-oct20} 

$ $
\begin{enumerate}[\rm(a)]
\item\label{17-oct20.a} The maps and modules of {\rm(\ref{17ll})} form a map of complexes.
\item\label{17-oct20.b} The maps and modules 
$$0\to
\ov{L_0}\xrightarrow{\ov{\rho_0}}  
\ov{L_1}\xrightarrow{\ov{\rho_1}}
\ov{L_2}\xrightarrow{\ov{\rho_2}}
\ov{L_3}\xrightarrow{\ov{\rho_3}}
\ov{L_4}\xrightarrow{\ov{\rho_4}}
\ov{L_5}\to 0$$
form a complex.
\end{enumerate}
\end{lemma}

\begin{proof} We compute in $P$. Keep in mind that the image of $f$ in $\Pbar$ is zero. Observe that
\begingroup\allowdisplaybreaks 
\begin{align*}
&(\ell_1\circ \rho_0)_{1,1}=\beta'_0\circ\a_0 =f\cdot \id_{K_0},&&\text{by (\ref{f-oct-31});}
\\
&(\rho_0\circ \ell_1-\ell_2\circ \rho_1)_{1,1}=
\a_0\circ \beta'_0=f\cdot \id_{M_0},&&\text{by (\ref{f-oct-31});}\\
&(\rho_0\circ \ell_1-\ell_2\circ \rho_1)_{1,2}
=-\a_0\circ k_1+m_1\circ \a_1=0,&&\text{by (\ref{17.6});}\\
&(\rho_0\circ \ell_1-\ell_2\circ \rho_1)_{2,1}=0;\\
&(\rho_0\circ \ell_1-\ell_2\circ \rho_1)_{2,2}=
\beta_1'\circ \a_1-k_2\circ z_1\\
{}={}&r\beta_0(1)\cdot \id_{K_1}+z_0\circ k_1-k_2\circ z_1\\
{}={}&(r\beta_0(1)+k_1(\sigma))\cdot \id_{K_1}=f\cdot\id_{K_1},\intertext{by \ref{17.3}.(\ref{17.3.a}), (\ref{17.6}), and \ref{17*.3}.(\ref{17*.3.a});}
&(\rho_1\circ \ell_2-\ell_3\circ \rho_2)_{1,1}=-\a_1\circ\beta_1'-m_2\circ (rX-w_1)\\
{}={}&-r(\a_1\circ \beta_1+m_2\circ X)-(w_0\circ \a_0\circ m_1-m_2\circ w_1)\\
{}={}&-(r\beta_0(1)+k_1(\sigma))\cdot \id_{M_1}=-f\cdot\id_{M_1},\intertext{by  \ref{17*.4}, \ref{17.9}.(\ref{17.9.b}), and
\ref{17*.3}.(\ref{17*.3.b});}
&(\rho_1\circ \ell_2-\ell_3\circ \rho_2)_{1,2}=\a_1\circ k_2-m_2\circ\a_2=0,&&\text{by (\ref{17.6});}\\
&(\rho_1\circ \ell_2-\ell_3\circ \rho_2)_{2,1}=-z_1\circ \beta_1'-\beta_2'\circ(rX-w_1)\\
{}={}&r^2\beta_2\circ X+r(-z_1\circ \beta_1+\beta_2\circ w_1)-z_1\circ z_0\circ m_1=0,\intertext{by \ref{17*.2}.(\ref{17*.2.b}) and \ref{17*.4};}
&(\rho_1\circ \ell_2-\ell_3\circ \rho_2)_{2,2}=z_1\circ k_2-k_3\circ z_2-r\beta_2\circ \a_2\\
{}={}&-(k_1(\sigma)+r\beta_0(1))\cdot \id_{K_2}=-f\cdot \id_{K_2},\intertext{by \ref{17*.3}.(\ref{17*.3.a}) and \ref{17.3}.(\ref{17.3.a});}
&(\rho_2\circ \ell_3-\ell_4\circ \rho_3)_{1,1}\\
{}={}&r(X\circ m_2+\a_2\circ \beta_2+m_3\circ X^\dagger)-w_1\circ m_2+m_3\circ w_2\\
{}={}&(r\beta_0(1)+k_1(\sigma))\cdot \id_{M_2}=f\cdot \id_{M_2},\intertext{by \ref{17.9}.(\ref{17.9.c}) and \ref{17*.3}.(\ref{17*.3.b});}
&(\rho_2\circ \ell_3-\ell_4\circ \rho_3)_{1,2}=-\a_2\circ k_3+m_3\circ\a_3=0,&&\text{by (\ref{17.6});}\\
&(\rho_2\circ \ell_3-\ell_4\circ \rho_3)_{2,1}
\\
{}={}&r^2(\beta_3\circ X^\dagger)+r(-z_2\circ \beta_2+\beta_3\circ w_2)=0,\intertext{by \ref{17*.2}.(\ref{17*.2.a}) and \ref{17*.4};}
&(\rho_2\circ \ell_3-\ell_4\circ \rho_3)_{2,2}
=r\cdot \beta_3\circ \a_3+z_2\circ k_3-k_4\circ z_3\\
{}={}&(r\beta_0(1)+k_1(\sigma))\cdot \id_{K_3}=f\cdot \id_{K_3},\intertext{by \ref{17.3}.(\ref{17.3.a}) and \ref{17*.3}.(\ref{17*.3.a});}
&(\rho_3\circ \ell_4-\ell_5\circ \rho_4)_{1,1}\\
{}={}&-r(X^\dagger\circ m_3+\a_3\circ \beta_3)-w_2\circ m_3+m_4\circ w_3\\
{}={}&-(r\beta_0(1)+k_1(\sigma))\cdot\id_{M_3}=-f\cdot \id_{M_3},\intertext{by \ref{17*.2}.(\ref{17*.2.d}) and \ref{17*.3}.(\ref{17*.3.b});}
&(\rho_3\circ \ell_4-\ell_5\circ \rho_4)_{1,2}=\a_3\circ k_4-m_4\circ \a_4=0,&&\text{by (\ref{17.6});}\\
&(\rho_3\circ \ell_4-\ell_5\circ \rho_4)_{2,1}
=r(-z_3\circ \beta_3+\beta_4\circ w_3)=0,&&\text{by \ref{17*.4};}\\
&(\rho_3\circ \ell_4-\ell_5\circ \rho_4)_{2,2}=z_3\circ k_4-r\beta_4\circ \a_4\\
{}={}&-(k_1(\sigma)+r\beta_0(1))\cdot \id_{K_4}=-f\cdot \id_{K_4},\intertext{by \ref{17*.3}.(\ref{17*.3.a}) and
\ref{17.3}.(\ref{17.3.a});}
&(\rho_4\circ \ell_5)_{1,1}=-w_3\circ m_4+r\a_4\circ \beta_4=f\cdot \id_{M_4},&&\text{by \ref{17.47}.(\ref{17.47.c});}\\
&(\rho_1\circ \rho_0)=0;\\
&(\rho_2\circ \rho_1)_{1,1}=0;\\
&(\rho_2\circ \rho_1)_{1,2}
=-rX\circ \a_1+w_1\circ \a_1-\a_2\circ z_1=0,\intertext{by \ref{17.9}.(\ref{17.9.a}) and \ref{17*.4};}
&(\rho_2\circ \rho_1)_{2,1}=0;\\
&(\rho_2\circ \rho_1)_{2,2}=z_2\circ z_1=0,&&\text{by  \ref{17*.4};}\\
&(\rho_3\circ \rho_2)_{1,1}=
(-rX^\dagger-w_2)\circ(rX-w_1)\\
{}={}&-r^2X^\dagger\circ X+r(X^\dagger\circ w_1-w_2\circ X)+w_2\circ w_1=0, \intertext{by \ref{17.9}.(\ref{17.9.d}), Lemma~\ref{17*.5'}.(\ref{17*.5'.a}), and \ref{17*.4};}
&(\rho_3\circ \rho_2)_{1,2}
=-rX^\dagger\circ\a_2-w_2\circ\a_2+\a_3\circ z_2=0,\intertext{by \ref{17.9}.(\ref{17.9.e})  and \ref{17*.4};}
&(\rho_3\circ \rho_2)_{2,1}=0;\\
&(\rho_3\circ \rho_2)_{2,2}=z_3\circ z_2=0, &&\text{by \ref{17*.4};}\\
&(\rho_4\circ \rho_3)_{1,1} =rw_3\circ X^\dagger+w_3\circ w_2=0,\intertext{by \ref{17*.2}.(\ref{17*.2.c}) and \ref{17*.4}; and}
&(\rho_4\circ \rho_3)_{1,2} =w_3\circ \a_3-\a_4\circ z_3=0,&&\text{by  \ref{17*.4}.} \end{align*}
\endgroup
\vskip-25pt\end{proof}

\section{Other interpretations of $X$.}\label{10}

\begin{chunk}Adopt the notation of {\rm \ref{data17}} and  {\rm\ref{17.42}}.
 Fix elements $\e_1,\e_2,\e_3,\e_4$ in $K_1$ with $$[\e_1\w\e_2\w\e_3\w\e_4]_K=1.$$ It is not difficult to see that the homomorphism $X:M_1\to M_2$ satisfies \ref{17.9}.(\ref{17.9.b}) if and only if
$(m_2\circ X)(\theta_1)$ is equal to
$$\begin{cases}
\phantom{+}[\a_1(\e_1)\a_1(\e_2)\a_1(\e_3)\a_1(\e_4)]_M\cdot \theta_1 -[\theta_1\a_1(\e_2)\a_1(\e_3)\a_1(\e_4)]_M\cdot  \a_1(\e_1)\\
+[\theta_1\a_1(\e_1)\a_1(\e_3)\a_1(\e_4)]_M\cdot  \a_1(\e_2)
-[\theta_1\a_1(\e_1)\a_1(\e_2)\a_1(\e_4)]_M\cdot  \a_1(\e_3)\\
+[\theta_1\a_1(\e_1)\a_1(\e_2)\a_1(\e_3)]_M\cdot  \a_1(\e_4)\end{cases}$$
and $X$ satisfies \ref{17.9}.(\ref{17.9.c}) if and only if $\Big((X\circ m_2)(\theta_2)\Big)(\theta_2')+\Big((X\circ m_2)(\theta_2')\Big)(\theta_2)$ is equal to
$$\begin{cases}
-[\theta_2\a_1(\e_3)\a_1(\e_4)]_M \cdot \a_1(\e_1)\a_1(\e_2)\theta_2'+[\theta_2\a_1(\e_2)\a_1(\e_4)]_M\cdot  \a_1(\e_1)\a_1(\e_3)\theta_2'\\
-[\theta_2\a_1(\e_2)\a_1(\e_3)]_M \cdot \a_1(\e_1)\a_1(\e_4)\theta_2'
-[\theta_2\a_1(\e_1)\a_1(\e_2)]_M\cdot  \a_1(\e_3)\a_1(\e_4)\theta_2'\\
+[\theta_2\a_1(\e_1)\a_1(\e_3)]_M \cdot \a_1(\e_2)\a_1(\e_4)\theta_2'
-[\theta_2\a_1(\e_1)\a_1(\e_4)]_M\cdot  \a_1(\e_2)\a_1(\e_3)\theta_2'\\
+[\a_1(\e_1)\a_1(\e_2)\a_1(\e_3)\a_1(\e_4)]_M\cdot \theta_2\cdot \theta_2'.\end{cases}$$
Maps $X$ with the above two  properties are considered in \cite{Sl93,Ku-aci-dg}. In particular, in the language of  \cite[Def.~1.3]{Ku-aci-dg}, the map $M_2\t M_1\to P$, which is given by $$\theta_2\t \theta_1\mapsto [X(\theta_1)\cdot \theta_2]_M,$$ is called a ``partial higher order multiplication'' on $M$. (The higher order multiplication is called partial, rather than complete, because the element $$\a_1(\e_1)\w \a_1(\e_2)\w \a_1(\e_3)\w \a_1(\e_4)$$ of $\bw^4M_1$ is held fixed, rather than allowed to be arbitrary.) The papers \cite{Sl93,Ku-aci-dg} use higher order multiplication to prove that if $P$ is a local ring in which two is a unit, then the minimal resolution of the almost complete intersection ring $P/(\mfK,f)$, by free $P$-modules, is a $\DG$-algebra. In particular, the paper \cite{Sl93} proves that if $P$ is a local ring in which two is a unit, then $M$ has a complete higher order multiplication. In the present paper, we are able to obtain higher order multiplication over any commutative Noetherian ring; we do not require that two be a unit or that the ring be local. The present paper makes significant use of divided powers; see, in particular, the complex $B$ of Definition~\ref{5-21.1}. The concept of divided powers   barely appears in \cite{Sl93,Ku-aci-dg}.
 In the present paper we did not consider complete higher order multiplications.
\end{chunk}

\begin{chunk} The map $X$ of Theorem~\ref{17.9} gives  the following null homotopy:
$$\xymatrix{
0\ar[r]&M_4\ar[rr]^{m_4}\ar[dd]_{w_4=0}&&M_3\ar[rr]^{m_3}\ar[dd]^{w_3}\ar[ddll]_(.4){h_3}&&M_2\ar[rr]^{m_2}\ar[dd]^{w_2}\ar[ddll]_(.4){h_2}&&M_1\ar[rr]^{m_1}\ar[dd]^{w_1}\ar[ddll]_(.4){h_1}&&M_0\ar[dd]_{w_0=0}\ar[ddll]_(.4){h_0}\\\\
0\ar[r]&M_4\ar[rr]^{m_4}&&M_3\ar[rr]^{m_3}&&M_2\ar[rr]^{m_2}&&M_1\ar[rr]^{m_1}&&M_0,}$$
where $w_i: M_i\to M_i$ is given by 
$$w_i(\theta_i)=\beta_0(1)\theta_i-(\a_i\circ \beta_i)\theta_i;$$
$h_0$ and $h_3$ are both zero; $h_1=X$; and $h_2=X^\dagger$.
\end{chunk}

\end{document}